\documentclass[a4paper,10pt,reqno]{amsart}

\usepackage{graphicx,color}
\usepackage{pinlabel-test}
\usepackage[centertags]{amsmath}
\usepackage{amsfonts,amssymb,amsthm,amsmath}
\usepackage{float,longtable}
\usepackage{a4wide}
\usepackage[latin1]{inputenc}
\usepackage{newlfont}
\usepackage{booktabs}
\usepackage{textcomp}
\usepackage{subfigure}

 \theoremstyle{plain}
 \begingroup
 \newtheorem{thm}{Theorem}[section]
 \newtheorem{lem}[thm]{Lemma}
 
 \newtheorem{cor}[thm]{Corollary}
 
 \endgroup

 \theoremstyle{definition}
 \begingroup

 \endgroup

 \theoremstyle{remark}
 \begingroup
 \newtheorem{rem}[thm]{Remark}
 \endgroup

 \numberwithin{equation}{section}

\let\E\eub

\newcommand{\R}{{\mathbb R}}

\def\BVloc{BV_{\mathrm{loc}}}

\def\un#1{\mathop{\textrm{#1}}}  
\def\weakto{\rightharpoonup}

\let\longweakto\longrightharpoonup
\long\def\drop#1{}
\def\wmm{\mu_n\boxtimes\mu_n}
\def\dom{\Omega}
\def\doms{{\Omega^2}}
\def\M{\mathcal M}

\DeclareRobustCommand{\1}[1]{\ensuremath \mathbf{1}_{\{#1\}}}
\DeclareMathOperator{\supp}{supp}
\let\x\xi
\def\d{\mathrm d}
\def\Lebesgue{\mathcal L}
\def\bx{\mathbf x}
\def\by{\mathbf y}

\title[Dislocation Density]
{Asymptotic behaviour of a pile-up of infinite walls\\ of edge dislocations}

\author[]{M.G.D. Geers$^2$}
\author[]{R.H.J. Peerlings$^2$}
\author[]{M.A. Peletier$^{3,4}$}
\author[]{L. Scardia$^{1,2,3}$}

\begin{document}

\begin{abstract}
We consider a system of parallel straight edge dislocations and we analyse its asymptotic behaviour in the limit of many dislocations. The dislocations are represented by points in a plane, and they are arranged in vertical \emph{walls}; each wall is free to move in the horizontal direction. The system is described by a discrete energy depending on the one-dimensional horizontal positions $x_i>0$ of the $n$ walls; the energy contains contributions from repulsive pairwise interactions between all walls, a global shear stress forcing the walls to the left, and a pinned wall at $x=0$ that prevents the walls from leaving through the left boundary.

We study the behaviour of the energy as the number $n$ of walls tends to infinity, and characterise this behaviour in terms of $\Gamma$-convergence. There are five different cases, depending on the asymptotic behaviour of the single dimensionless parameter $\beta_n$, corresponding to $\beta_n\ll 1/n$, $1/n\ll\beta_n \ll 1$, and $\beta_n\gg1$, and the two critical regimes $\beta_n\sim 1/n$ and $\beta_n\sim 1$.
As a consequence we obtain characterisations of the limiting behaviour of stationary states in each of these five regimes.

The results shed new light on the open problem of upscaling large numbers of dislocations. We show how various existing upscaled models arise as special cases of the theorems of this paper. The wide variety of behaviour suggests that upscaled models should incorporate more information than just dislocation densities. This additional information is encoded in the limit of the dimensionless parameter $\beta_n$.
\end{abstract}

\maketitle

\footnotetext[1]{Materials innovation institute (M2i)}
\footnotetext[2]{Department of Mechanical Engineering, Technische Universiteit Eindhoven}
\footnotetext[3]{Department of Mathematics and Computer Sciences, Technische Universiteit Eindhoven}
\footnotetext[4]{Institute for Complex Molecular Systems, Technische Universiteit Eindhoven}

\section{Introduction}

\subsection{Dislocation plasticity}


One of the hard open problems in engineering is the upscaling of large numbers of \emph{dislocations}. Dislocations are defects in the crystal lattice of a metal, and their collective motion gives rise to macroscopic permanent or \emph{plastic} deformation.

For systems of millimeter-size or larger there is a fairly complete theory of macroscopic plasticity, in which dislocations are not modelled explictly (see e.g.~\cite{Hill66,HillRice72,HillHavner82,AsaroRice77}). For smaller systems, however, the so-called \emph{size effect}~\cite{Hall51a,Petch53,FleckMullerAshbyHutchinson94} suggests that it is necessary to take the distribution of dislocations into account. In this point of view the size effect arises when the length scale of the system becomes similar to the typical scale at which the \emph{dislocation density} varies.

To address these small-scale effects a number of competing (mainly phenomenological) \emph{dislocation density models} have been derived by upscaling large numbers of dislocations (e.g.~\cite{DengEl-Azab09,El-Azab00,Groma97,GromaBalogh99,GromaCsikorZaiser03,LimkumnerdVan-der-Giessen08,RoyAcharya06}). The unknowns in this type of model are various types of dislocation densities, whose evolution in time is described via conservation laws equipped with constitutive laws both for the velocity of the dislocations and for their interaction.

The use of densities (as opposed to keeping track of the behaviour of each dislocation) seems reasonable, since the typical number of dislocations in a portion of metal is huge. For topological reasons dislocations are \emph{curves} in three-dimensional space, and therefore the density of  dislocations has dimensions of $\text{length}/\text{volume}$ or $\mathrm m^{-2}$. A  dislocation density of $10^{15}\un{m}^{-2}$ (typical for cold-rolled metal~\cite[p.~20]{HullBacon01}) translates into $1000\un{km}$ of dislocation curve in a cubic millimeter of metal. This high number explains the interest in avoiding the description of the individual behaviour of each dislocation, and focussing on the collective behaviour instead. It also explains the general belief that this should be possible.

\medskip

The research done here, however, suggests that the situation is more subtle. It was triggered by the earlier study~\cite{RoyPeerlingsGeersKasyanyuk08}. The outcome of~\cite{RoyPeerlingsGeersKasyanyuk08} and the results in this paper suggest that the dislocation density alone is \emph{not capable} of describing the evolution of large numbers of dislocations.
To put it succinctly, a density simply does not contain enough information to characterise the behaviour of the system, even in aggregate form. This is because the density only characterizes the local \emph{number} of dislocations per unit area, and needs to be supplemented with more information on their \emph{spatial arrangement} in order to give a satisfactory answer.
We show below how this point arises from the results of this paper.

A separate reason for the analysis of this paper is the uncommon form of the energy.
Although it is a simple two-point interaction energy in one spatial dimension, the  behaviour in the many-dislocation limit does not fit into any of the standard cases as described e.g.\ in~\cite{BraidesMaria06}. This is due to the combination of all-neighbours-interaction (each pair interaction is counted, regardless of distance), and an interaction potential that is globally repulsive.

\subsection{The model of this paper}
We consider a system of pure edge dislocations whose dislocation lines are straight and parallel to one another as in~\cite{RoyPeerlingsGeersKasyanyuk08}. These dislocations can be modelled as points in the plane orthogonal to the direction of the dislocation lines, and this identification has been done systematically in the literature. The \emph{slip planes} are horizontal, i.e.\ parallel to the $\tilde{x}$-coordinate, which implies that the dislocations can only move in the horizontal direction (see Figure~\ref{Wall}). In addition, the dislocations are organized in vertical \emph{walls} with a uniform spacing of size~$h$ (in $\un m$). In the model below there will be a finite number~$n$ of such walls, which each will extend indefinitely in the vertical direction. The total degrees of freedom of the system are therefore the horizontal positions $0\leq\tilde x_1 \leq \dots\leq \tilde x_n$ (in $\un m$) of the walls. A constant global shear stress forces the walls towards a fixed barrier which is modelled as an infinite wall of pinned dislocations at $\tilde{x}_0=0$.
\begin{figure}[htbp]
\labellist
\pinlabel \small$\tilde x_0$ at 90 208
\pinlabel \small$\tilde x_1$ at 124 208
\pinlabel \small$\tilde x_2$ at 184 208
\pinlabel $\tilde x$ at 431 208
\pinlabel $h$ [l] at 421 307
\pinlabel $\sigma$ [b] at 247 410
\pinlabel $\sigma$ [t] at 241 90
\pinlabel {\small slip planes} [l] at 508 220
\endlabellist
\begin{center}
\includegraphics[width=3in]{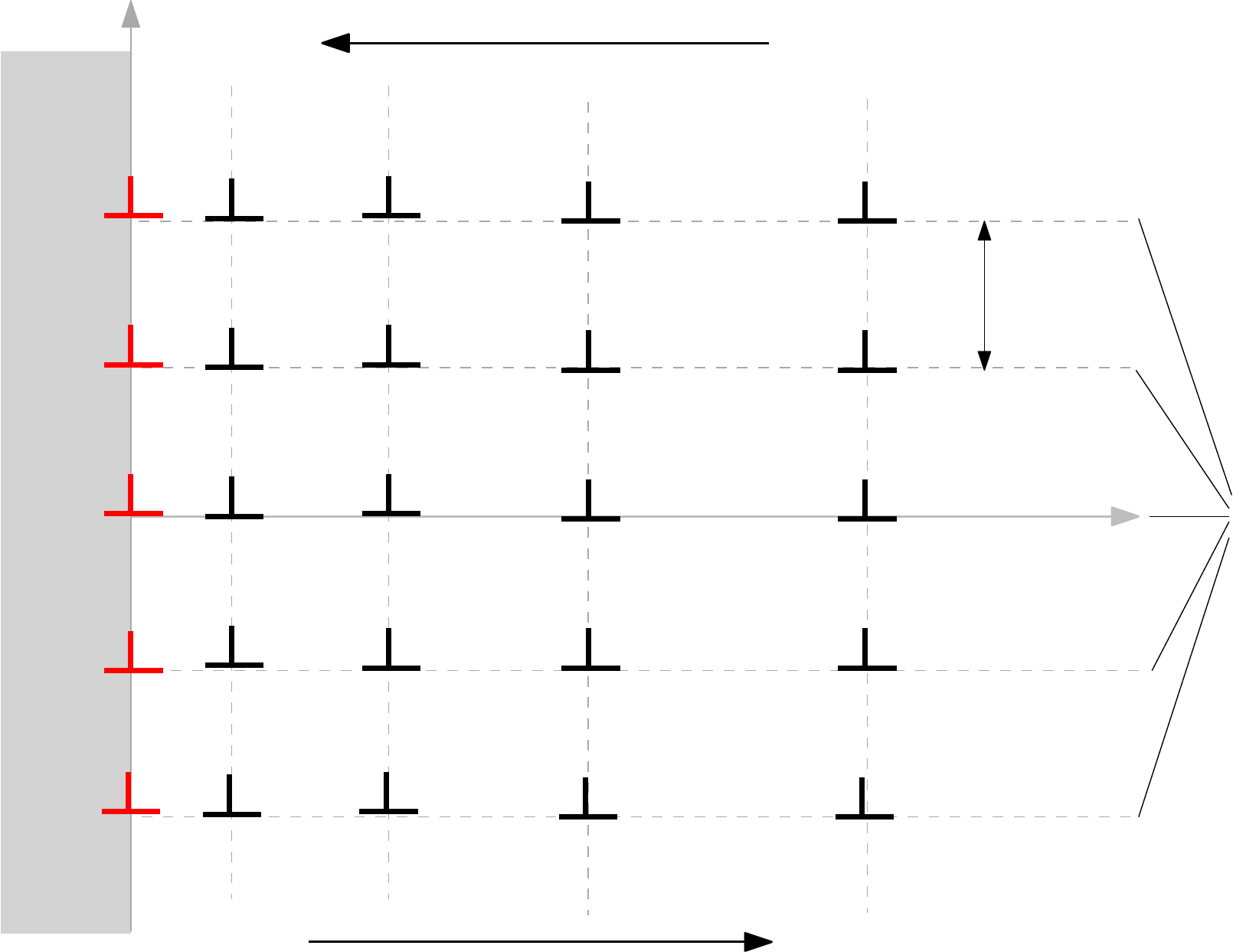}
\end{center}
\caption{The dislocation configuration considered in this paper. Infinite, vertical \emph{walls} of equispaced dislocations are free to move in the horizontal direction. A wall of fixed dislocations is pinned at $\tilde x_0=0$ and acts as a repellent.}
\label{Wall}
\end{figure}
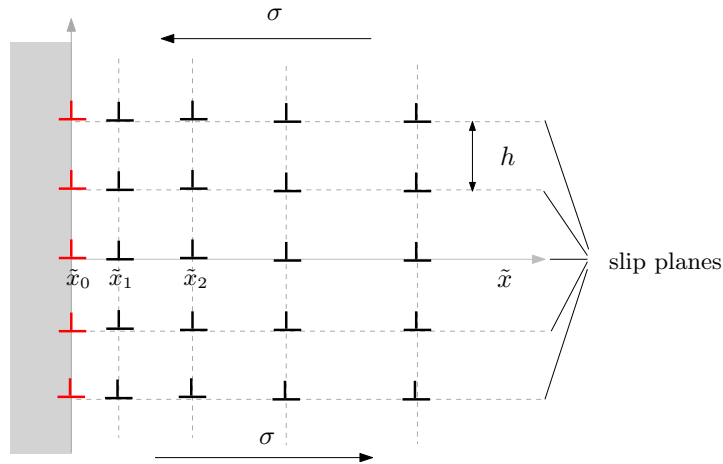

We assume that the dislocations are spaced significantly farther apart than the atomic lattice spacing, which implies that the interactions between dislocation walls are well described by conventional formulae based on linear elasticity.

Given these assumptions, the system is driven by the discrete energy
\begin{equation}
\label{discreteenergy}
\E(\tilde{x}_1,\dots,\tilde{x}_n) = \frac K2\sum_{i=1}^n\sum_{\stackrel{j=0}{j\neq i}}^n V\left(\frac{\tilde x_i-\tilde x_j}{h}\right) +\sigma \sum_{i=1}^n \tilde x_i,
\end{equation}
where $ K := G b \pi/2(1-\nu)$, $G\; [\un{Pa}]$ is the shear modulus, $b \;[ \un m]$ the length of the Burgers vector, $\nu\;[1]$  the Poisson ratio of the material, and $\sigma\;[\un{Pa}]$ the imposed shear stress.
The (dimensionless) interaction energy  $V$ is
\begin{equation}
V(s):= \frac{1}{\pi}s\coth \pi s - \frac{1}{\pi^2}\log (2\sinh\pi s)\label{defV}
= \frac{2}{\pi}\frac{|s|}{(e^{2\pi |s|}-1)} - \frac{1}{\pi^2}\log(1-e^{-2\pi |s|}),
\end{equation}
and its derivative is the (dimensionless) force exerted by a wall on another wall at distance $s$,
\[
V'(s) = -\frac{s}{\sinh^2(\pi s)}.
\]

\begin{figure}[h]
\labellist
\pinlabel $V(s)$ [tr] at 170 100
\pinlabel $s$ at 345 -10
\pinlabel $\sim -\dfrac1{\pi^2}\log|s|$ [tl] at 180 220
\pinlabel $\sim\dfrac2\pi|s|e^{-2\pi|s|}$ [bl] at 240 20
\endlabellist
\begin{center}
\includegraphics[width=4in]{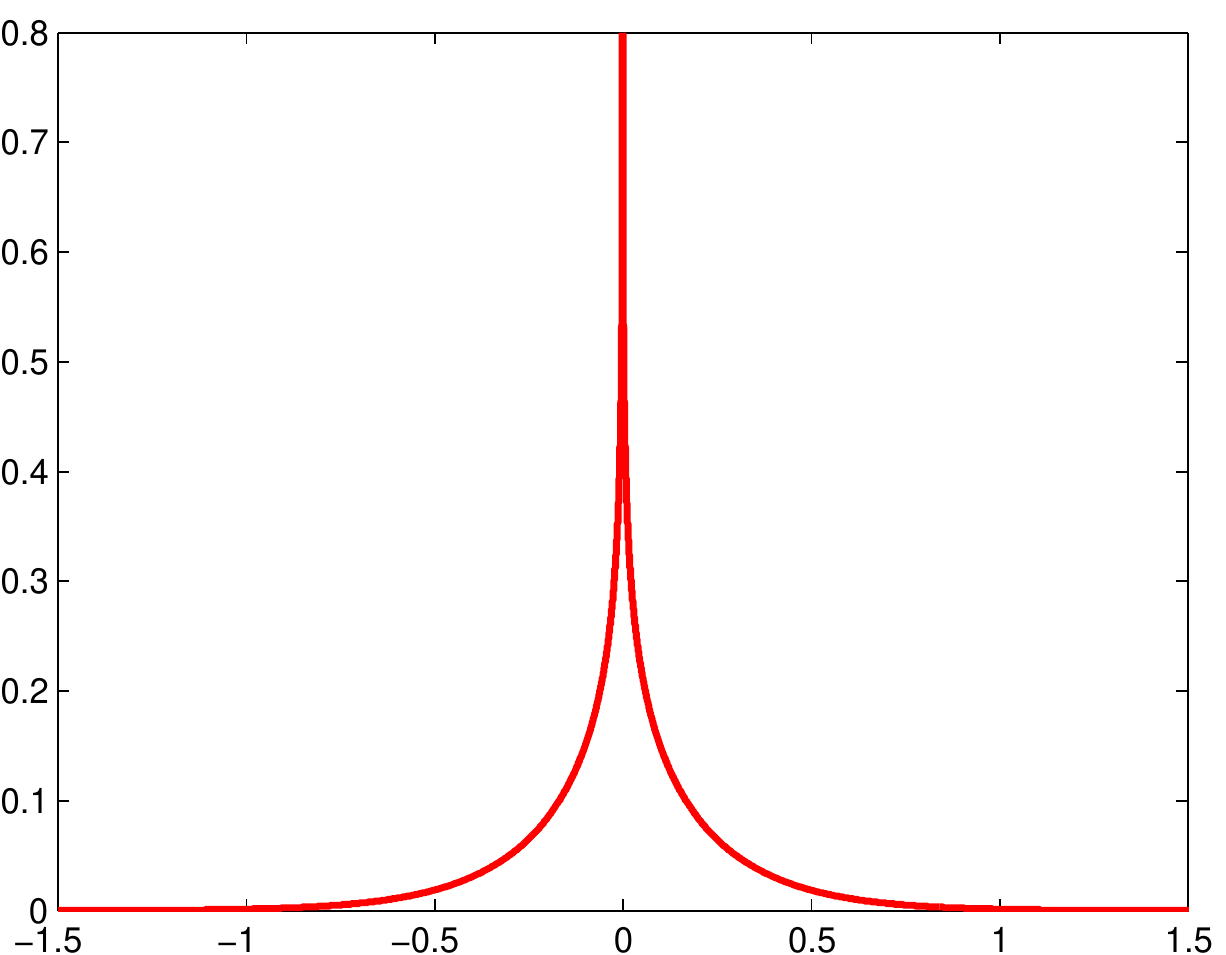}
\end{center}
\caption{The interaction energy  $V$.}
\label{PlotV}
\end{figure}
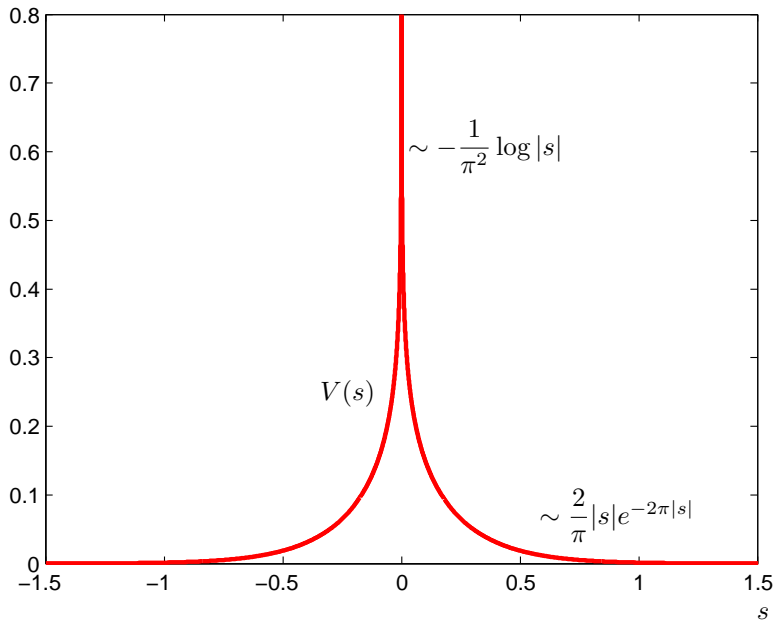

The first term in the discrete energy $\mathcal{E}$ in \eqref{discreteenergy} is fully repulsive: each pair of walls repels each other, with a potential that diverges logarithmically as the walls approach each other (see Figure~\ref{PlotV}). The second term of the energy, accounting for the global shear stress, drives the walls to the left. The repelling nature of the left barrier is implemented by pinning a wall at $\tilde x_0=0$.

Stationary points of this energy are equilibria of the mechanical system, and under the assumption of a linear drag relation (see e.g.~\cite[Sec.~3.5]{HullBacon01}) the evolution of the system is a gradient flow of this energy.

\medskip

Although the model is highly idealised, it has a number of properties that make it both interesting and not unrealistic. The fact that multiple dislocations move along exactly the same slip plane is natural, because of the way they are generated from \emph{Frank-Read sources} (e.g.~\cite[Sec.~8.6]{HullBacon01}). Moreover, although the assumption of an arrangement in equispaced vertical walls is clearly an idealisation, it is on the other hand not unrealistic since equispaced vertical walls are minimal-energy configurations. Walls of edge dislocations are locally stable, in the sense that if one of the dislocations deviates from its wall position, either horizontally or vertically, it experiences a restoring force from the other dislocations that pushes it back. Finally, the vertical organization in walls is also justified by correlation functions calculated from numerical simulations~(e.g.~\cite{ZaiserMiguelGroma01}).

Another interesting aspect of this model is that existing, phenomenological dislocation-density models can be applied to it to give predictions of the upscaled behaviour---which can then be tested against the rigorous results of this paper. In Section~\ref{sec:consequences} we discuss three of these, whose predictions for this system are summarized in Table \ref{Table}:

\begin{table}[H]\label{Table}
\renewcommand{\arraystretch}{1.2}
\begin{tabular}{lcc}
\toprule
Reference & Stationary state&$\sigma_{\mathrm{int}}$\\\midrule
Head \& Louat 1955~\cite{HeadLouat55} & $\rho(x) =\sqrt{\frac{C-x}{x}}$&\\
Groma, Csikor \& Zaiser 2003~\cite{GromaCsikorZaiser03}  & $\rho(x) = Ce^{-\widehat \sigma x}$ & $-\partial_x\rho/\rho$\\
Evers, Brekelmans and Geers 2004~\cite{EversBrekelmansGeers04}
   & $\rho(x) = C-\widehat\sigma x$& $-\partial_x\rho$\\
\bottomrule\\
\end{tabular}
\caption{Various predictions for the limiting behaviour of this system of dislocations. The system is characterized by the density $\rho$ of dislocations; $\sigma_{\mathrm{int}}$ is the prediction of the stress field generated by $\rho$. Parameters have been absorbed into $C$ and~$\widehat \sigma$ for simplicity.
We give a full discussion in Section~\ref{sec:consequences}.}
\end{table}
As we  show below, the results of this paper allow us to make sense of the three different predictions in this table.

\subsection{Main result}

The main mathematical result of this paper is the characterization of the limit behaviour of $\E$ as $n\to\infty$. Since this behaviour depends strongly on the assumptions on the behaviour of the whole set of other parameters  in this system, that is $h$, $K$ (or $G$, $b$, and $\nu$), and $\sigma$, we assume that {all parameters depend on $n$}. In fact the parameter space is only one-dimensional, since the problem can be rescaled to depend only on the single dimensionless parameter
\begin{equation}\label{betan}
\beta_n := \sqrt{\frac{K_n}{n\sigma_n h_n}} = \sqrt{\frac{\pi G_nb_n}{2n(1-\nu_n)\sigma_n h_n}}.
\end{equation}
In mechanical terms, $\beta_n$ measures the elastic properties of the medium (described by $K_n$) in comparison with the strength of the pile-up driving force $\sigma_n$. Large $\beta_n$, therefore, corresponds to weak forcing, and small $\beta_n$ to strong forcing

We characterize the limiting behaviour of the system by proving five $\Gamma$-convergence results, for five regimes of behaviour of $\beta_n$ as $n \to\infty$, after an appropriate rescaling of $\E$ and $(\tilde x_1,\dots,\tilde x_n)$ (rescalings that lead to the functionals $E^{(k)}_n(x_1,\dots,x_n)$, for $k\in\{1,2,3,4,5\}$, defined in Theorem \ref{th:main}).
A consequence of $\Gamma$-convergence is the convergence of minimizers.

Both the $\Gamma$-convergence of the energy and the convergence of minimizers depend on a concept of convergence for the set of wall positions $(\tilde x_i)_{i=1}^n$, or their rescaled versions $(x_i)_{i=1}^n$, as $n\to\infty$. A natural concept of convergence for such a system of wall positions  is weak convergence  of the corresponding empirical measures (which we prove being equivalent to the weak convergence of the linear interpolations of the wall positions in the space of functions with bounded variation, see Theorem \ref{compactness}).
For a vector $x\in \R^n$ define the \emph{empirical measure} as
\begin{equation}\label{measure}
\mu_n = \frac1n \sum_{i=1}^n \delta_{x_i}.
\end{equation}
The \emph{weak convergence} of $\mu_n$ to $\mu$, written as $\mu_n\longweakto \mu$, is defined by
\[
\int_{\Omega} \varphi(y)\, \mu_n(dy) \stackrel {n\to\infty} \longrightarrow
\int_{\Omega} \varphi(y)\, \mu(dy)\qquad\text{for all continuous and bounded }\varphi,
\]
where $\Omega:=[0,\infty)$. This is the concept of convergence that we use in this paper.

\begin{thm}[Asymptotic behaviour of $\E$]
\label{th:main}
In each of the cases below, boundedness of the functional $E_n^{(k)}$ implies that the empirical measures $\mu_n := \frac1n \sum_{i=1}^n{\delta_{ x_i}}$ are compact in the weak topology. In addition, the functional $E_n^{(k)}$ $\Gamma$-converges to a functional $E^{(k)}$ with respect to the same weak topology.

Case 1: If $\beta_n\ll 1/n$, then define the rescaled positions $x_1,\dots, x_n$ in terms of the physical positions $\tilde x_1,\dots,\tilde x_n$ by
\begin{equation}
\label{rescaling-x:case1}
 x_i = \tilde x_i \frac{ \sigma_n}{nK_n}
\end{equation}
and define
\[
E_n^{(1)}(x_1,\dots x_n) := \frac1{n^2K_n}\, {\E(\tilde x_1,\dots,\tilde x_n)} + \frac1{2\pi^2} (\log 2\pi n^2\beta_n^2 -1).
\]
Then $E_n^{(1)}$ $\Gamma$-converges to
\[
E^{(1)}(\mu):= -\frac{1}{2\pi^2}\iint_{\Omega^2} \log|x-y|\,\mu(dy)\mu(dx)  + \int_{\Omega} x \,\mu(dx).
\]

\bigskip

Cases $(2{-}4)$: If $\beta_n\sim 1/n$, $1/n\ll \beta_n \ll 1$, or $\beta_n\sim 1$, then define the rescaled positions $x_1,\dots, x_n$ as
\begin{equation}
\label{def:rescaling_case_2}
x_i = \tilde x_i \sqrt{\frac{\sigma_n}{n K_nh_n}},
\end{equation}
and define
\[
E_n^{(2-4)}(x_1,\dots,x_n) := \frac1{\sqrt{n^3K_nh_n\sigma_n}}\; \E(\tilde x_1,\dots,\tilde x_n).
\]
Then $E_n^{(2-4)}$ $\Gamma$-converges to
\begin{alignat}3
E^{(2)}(\mu) &:=\frac c 2\iint_{\Omega^2} V\bigl(c(x-y)\bigr) \,\mu(dx)\mu(dy) + \int_{\Omega} x\mu(dx)& \quad\text{if }{n\beta_n\to c},
  \label{def:limitenergyE2}\\
E^{(3)}(\mu) &:= \left\{
  \begin{aligned}
    &\frac12 \Bigl(\int_\R V\Bigr) \int_{\Omega} \rho(x)^2\, dx+ \int_{\Omega} x\rho(x)\, dx
      &\hspace{.25cm} & \text{if }\mu(dx) = \rho(x)\, dx\\
    &+\infty &&\text{otherwise}
  \end{aligned}
  \ \right\}
    &\quad\text{if }{\frac1n \ll \beta_n\ll 1},
    \label{def:limitenergyE3}\\
E^{(4)}(\mu)&:=\left\{
  \begin{aligned}
    &c\int_{\Omega} V_{\mathrm{eff}}\Bigl(\frac c{\rho(x)}\Bigr)\,\rho(x)dx
       +\int_{\Omega} x\rho(x)\, dx
          &\hspace{.1cm} & \text{if }\mu(dx) = \rho(x)\, dx\\
    &+\infty &&\text{otherwise}
  \end{aligned}
  \ \right\}
 &\quad\text{if }{\beta_n\to c},\label{def:limitenergyE4}
\end{alignat}
where the function $V_{\mathrm{eff}}$ in \eqref{def:limitenergyE4} is defined as
\[
V_{\mathrm{eff}}(s) := \sum_{k=1}^\infty V(ks).
\]

\bigskip

Case 5: If $\beta_n\gg1$, then define the rescaled positions $x_1,\dots, x_n$ as
\[
 x_i = \tilde x_i \left(\frac{1}{2\pi}\,nh_n\log\left(\frac{2}{\pi}\,\frac{K_n}{nh_n\sigma_n}\right)\right)^{-1}
\]
and define
\[
E_n^{(5)}(x_1,\dots x_n) := \left(\frac{1}{2\pi}n^2 h_n \sigma_n\log\left(\frac{2}{\pi}
\frac{K_n}{nh_n\sigma_n}\right)\right)^{-1}\mathcal{E}(\tilde{x}_1,\dots,\tilde{x}_n).
\]
Then $E_n^{(5)}$ $\Gamma$-converges to
\begin{align}
E^{(5)}(\mu) &:= \left\{
  \begin{aligned}
    &\int_{\Omega} x\rho(x)\,dx
      &\qquad & \text{if } \mu(dx) = \rho(x)\, dx \,\,\, \text{and } \, \rho\leq 1 \quad \mathcal{L} \, \text{-a.e.}\\
    &+\infty &&\text{otherwise}
  \end{aligned}
  \ \right\},
\end{align}
where $\Lebesgue$ is the Lebesgue measure.
\end{thm}

The limiting energies have the nice property of strict convexity, either with respect to the linear structure in the space of measures, or in the sense of displacement convexity~\cite{McCann97}. This gives uniqueness of minimizers:

\begin{thm}[Existence and uniqueness of limiting minimizers]
\label{lem:uniqueness}
For each $k\in\{1,2,3,4,5\}$, $E^{(k)}$ has a unique minimizer in the set $\M(\Omega)$ of non-negative, unit-mass Borel measures on $\Omega$.
\end{thm}

As a consequence we can characterize the behaviour of sequences of minimizers:

\begin{cor}[Convergence of discrete minimizers]
\label{cor:minimizers-converge}
Let the asymptotic behaviour of $\beta_n$ be as in case $k\in\{1,2,3,4,5\}$ of Theorem~\ref{th:main}.
Let $(\tilde x^n_1,\dots,\tilde x^n_n)$ be a sequence of $n$-vectors such that
for each~$n$, $(\tilde x^n_1,\dots,\tilde x^n_n)$ is minimal for $\E$. Then, rescaling $(\tilde x^n_1,\dots,\tilde x^n_n)$ to $( x^n_1,\dots, x^n_n)$ as in Theorem~\ref{th:main}, the corresponding empirical measure $\mu_n = \frac1n \sum_{i=1}^n \delta_{x_i^n}$ converges weakly to the global minimizer of~$E^{(k)}$.
\end{cor}

The proof of Theorem \ref{th:main} is the subject of Section \ref{sec:proof}; Theorem~\ref{lem:uniqueness} and Corollary \ref{cor:minimizers-converge} are proved in Section \ref{sec:uniqueness}.

\medskip

Figures~\ref{DD1}--\ref{DC5} show some numerical examples of the matching between discrete and continuous energies. Note that the optimal discrete density $\rho_n$ plotted in the Figures below is defined for every $i=2,\dots,n-1$ as
$$
\rho_n(x_i):= \frac{2A_n}{x_{i+1}-x_{i-1}},
$$
where $(x_1,\dots,x_n)$ is the minimiser of the discrete energy $E_n^{(k)}$, for $k=1,\dots,5$ and $A_n$ is a normalization factor ensuring that the area below the linear interpolant of $\rho_n$ is one.

%
\begin{figure}[htbp]
\labellist
\pinlabel {\small Dislocation positions} [b] at 180 -18
\pinlabel \scriptsize $\rho_n$ at 328 242.5
\pinlabel \scriptsize $\mu$ at 326 256
\endlabellist
\begin{center}
\includegraphics[width=3.6in]{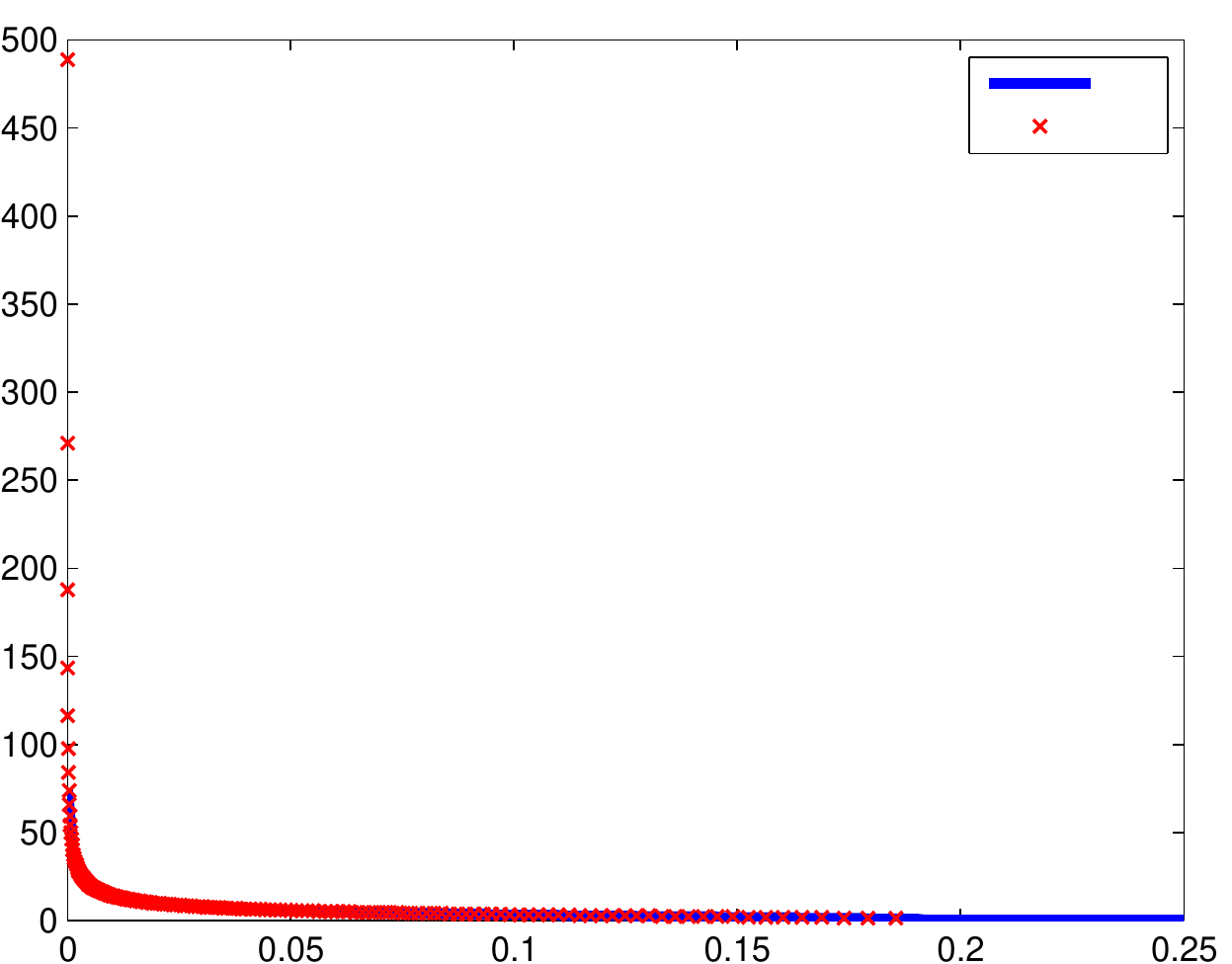}
\end{center}
\caption{Optimal densities relative to $E_n^{(1)}$ and $E^{(1)}$, for $n=150$ and $\beta_n = 6/(n\sqrt n)$.}
\label{DD1}
\end{figure}

\begin{figure}[htbp]
\labellist
\pinlabel {\small Dislocation positions} [b] at 180 -18
\pinlabel \scriptsize $\rho_n$ at 328 242.5
\pinlabel \scriptsize $\mu$ at 326 256
\endlabellist
\begin{center}
\includegraphics[width=3.6in]{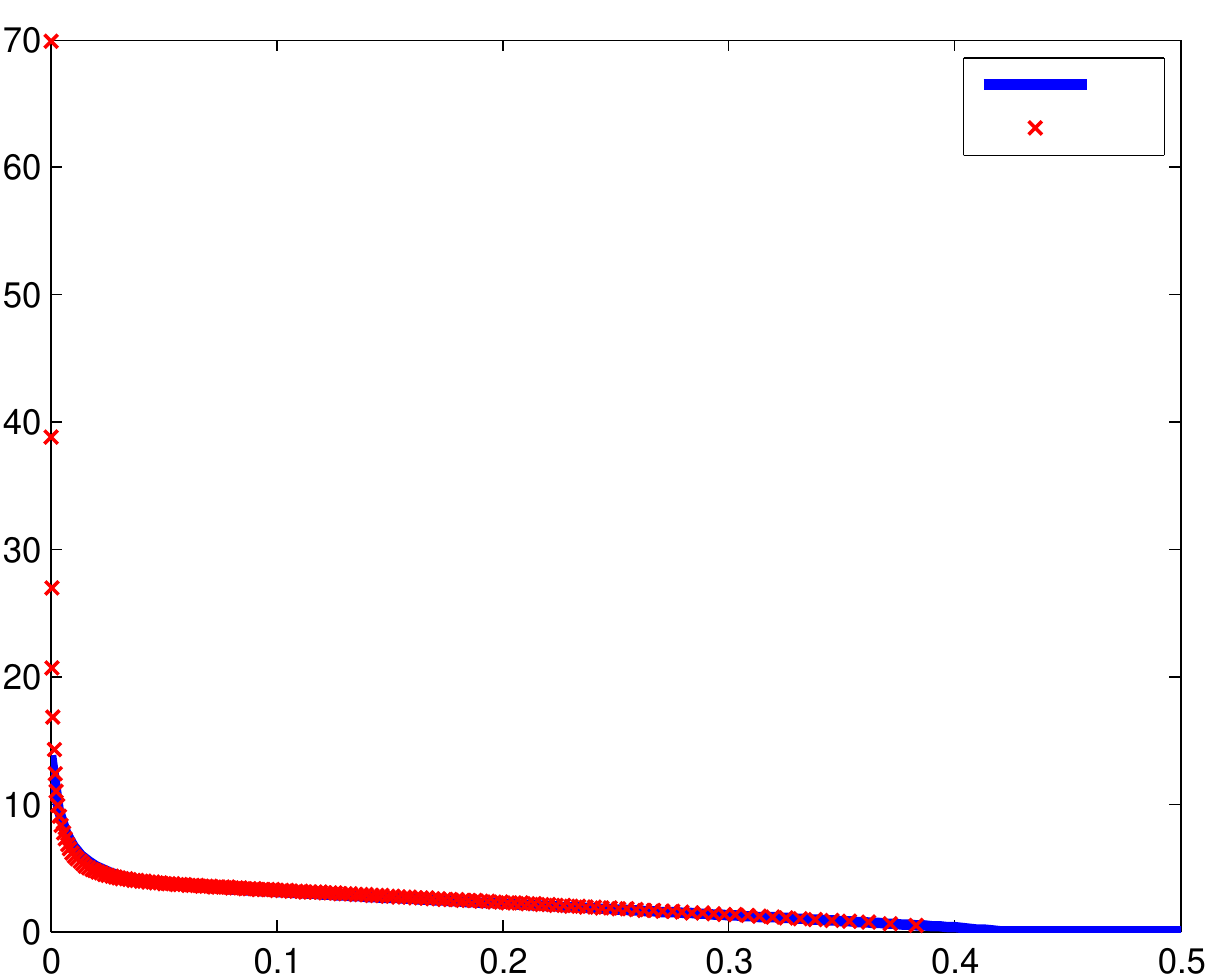}
\end{center}
\caption{Optimal densities relative to $E_n^{(2)}$ and $E^{(2)}$, for $n=150$ $\beta_n = 5/n$.}
\label{DD2}
\end{figure}

\begin{figure}[htbp]
\labellist
\pinlabel {\small Dislocation positions} [b] at 180 -18
\pinlabel \scriptsize $\rho_n$ at 328 242.5
\pinlabel \scriptsize $\mu$ at 326 256
\endlabellist
\begin{center}
\includegraphics[width=3.6in]{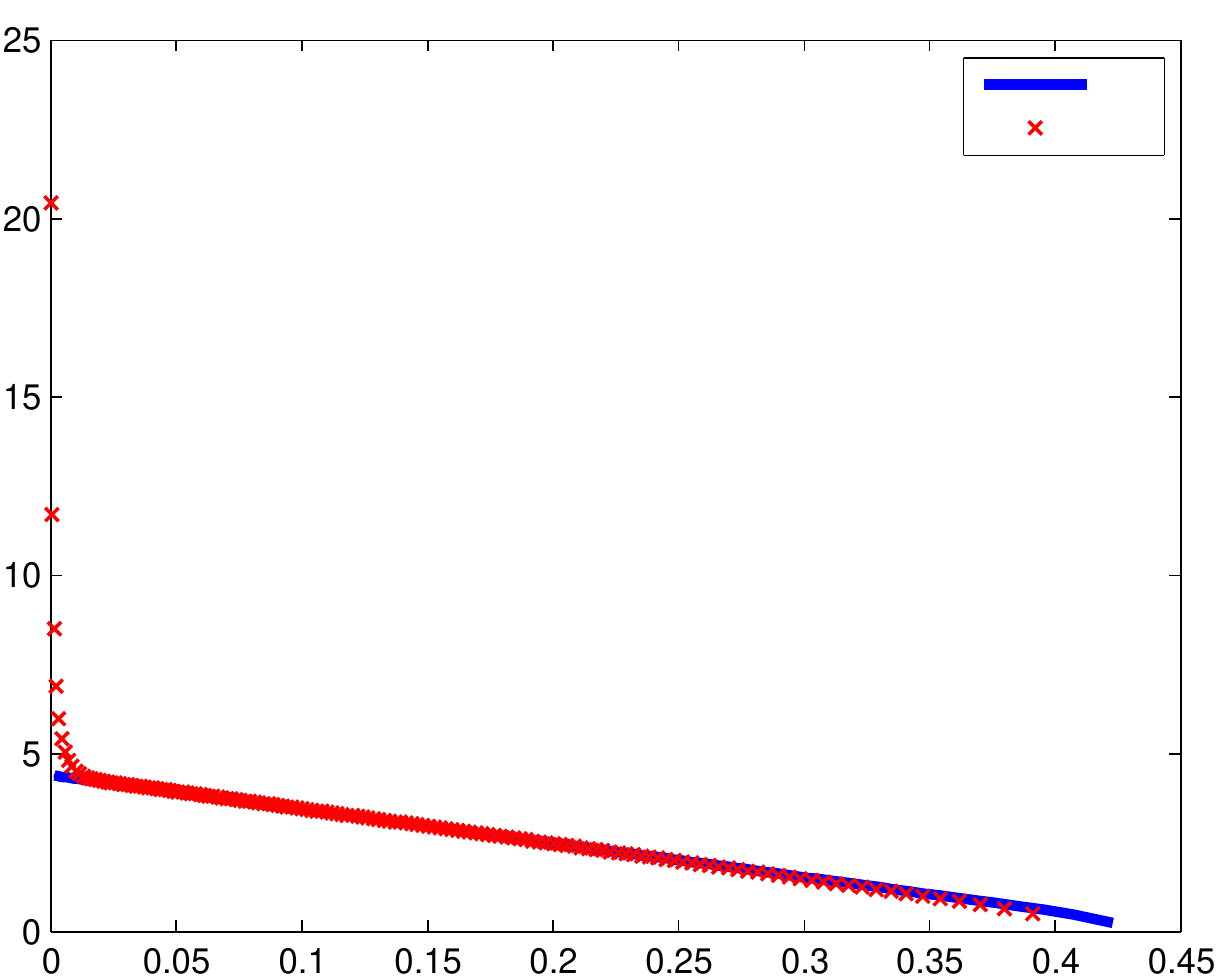}
\end{center}
\caption{Optimal densities relative to $E_n^{(3)}$ and $E^{(3)}$, for $n=150$ and $\beta_n = 1/\sqrt n=1/\sqrt{150}$.}
\label{DC3}
\end{figure}

\begin{figure}[htbp]
\labellist
\pinlabel {\small Dislocation positions} [b] at 180 -17
\pinlabel \scriptsize $\rho_n$ at 322 242.5
\pinlabel \scriptsize $\mu$ at 318 255
\endlabellist
\begin{center}
\includegraphics[width=3.6in]{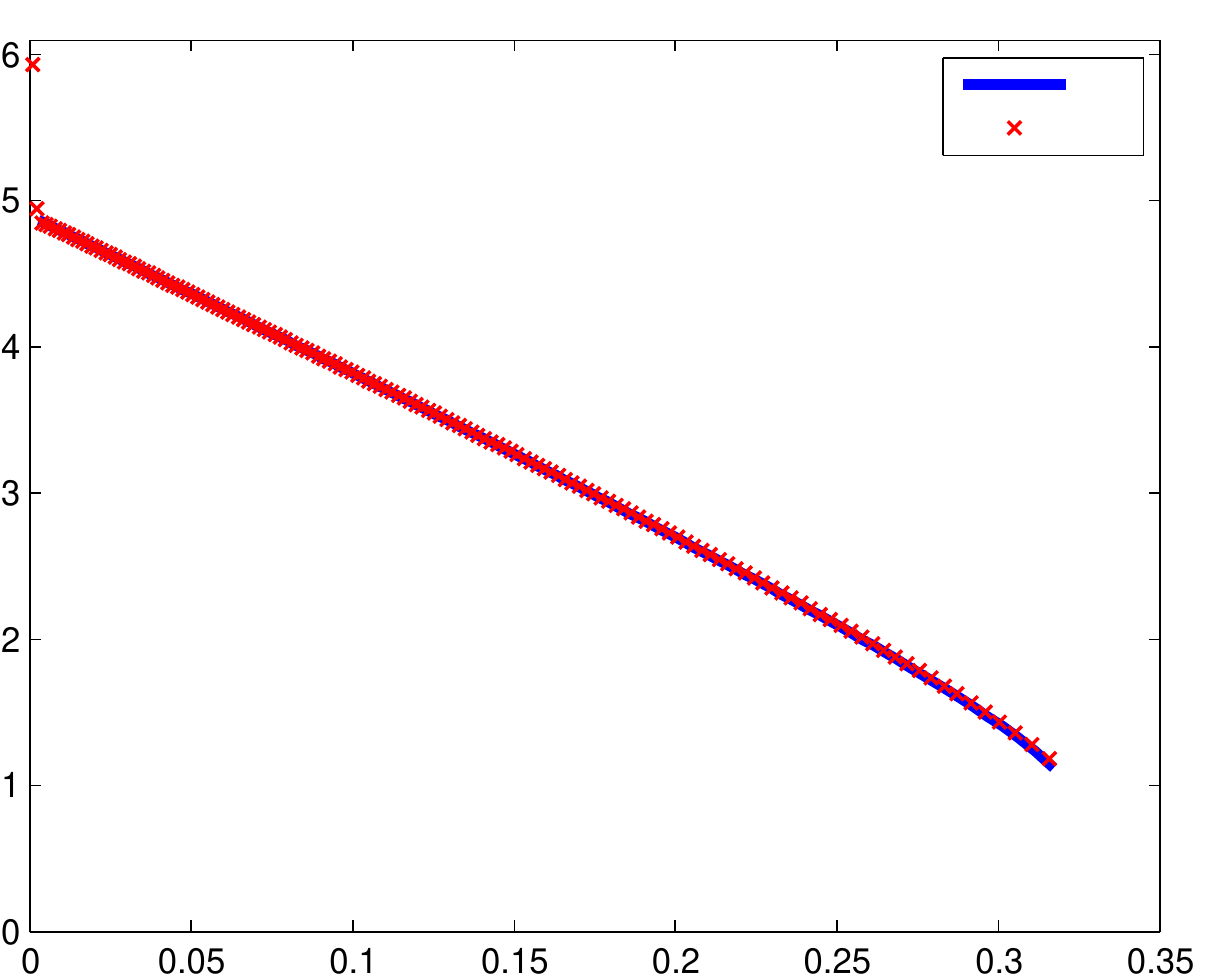}
\end{center}
\caption{Optimal densities relative to $E_n^{(4)}$ and $E^{(4)}$, with $n=150$ and $\beta_n =1$.}
\label{DC4}
\end{figure}

\begin{figure}[htbp]
\labellist
\pinlabel {\small Dislocation positions} [b] at 180 -17
\pinlabel \scriptsize $\rho_n$ at 328 242.5
\pinlabel \scriptsize $\mu$ at 326 255
\endlabellist
\begin{center}
\includegraphics[width=3.6in]{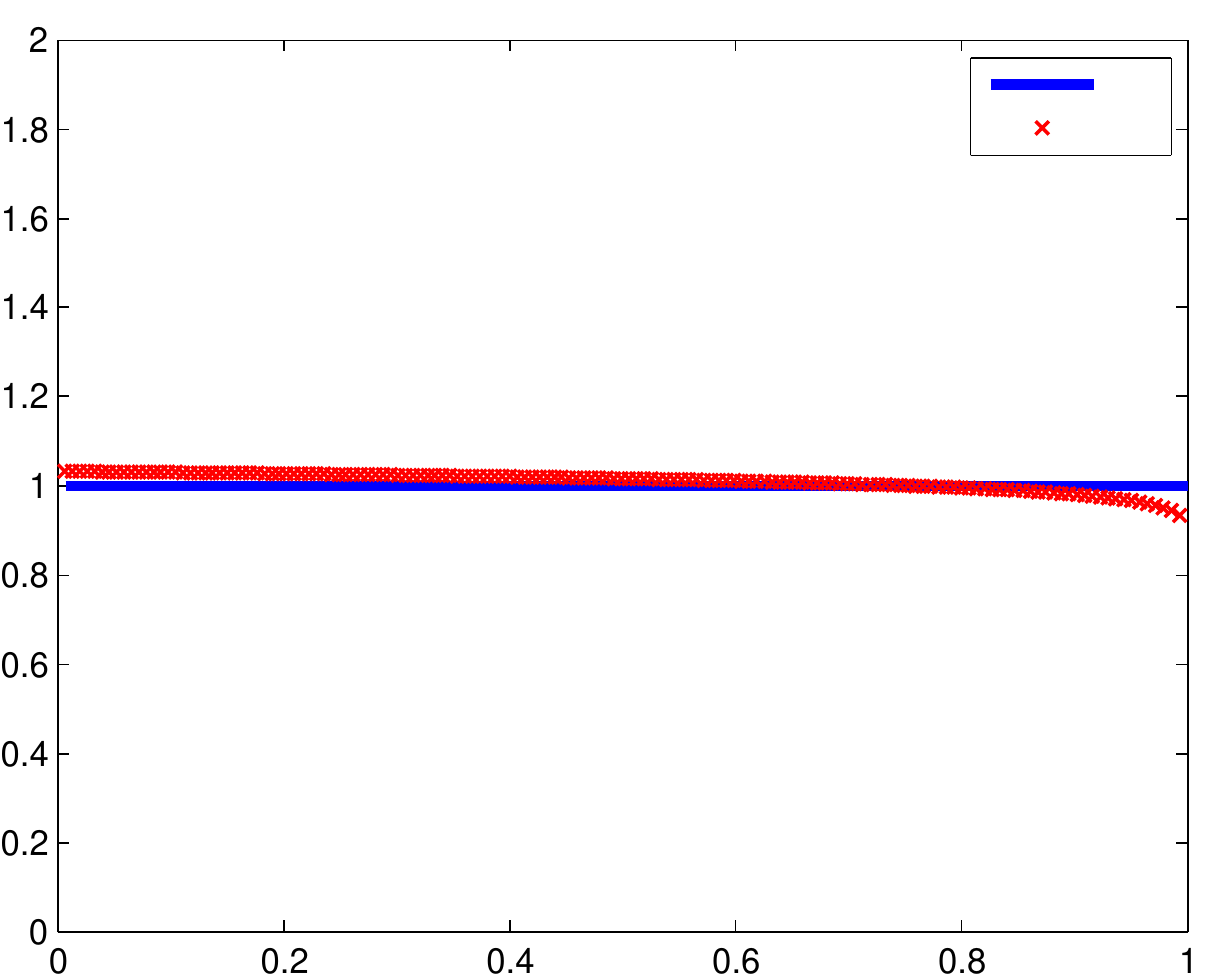}
\end{center}
\caption{Optimal densities relative to $E_n^{(5)}$ and $E^{(5)}$, with $n=200$ and $\beta_n =10^5$.}
\label{DC5}
\end{figure}

\subsection{Five regimes} The role of  $\beta_n$ and of the different asymptotic regimes can be understood as follows. Define the average dimensional distance between two walls (assuming $n$ even) as
\[
\Delta \tilde x := \frac{\tilde x_{n/2}}{n/2}.
\]
Note that $\tilde x_{n/2}$ is a `middle' wall, and therefore a reasonable indication of the size of the pileup. Assuming cases 2--4, we can then rewrite~\eqref{def:rescaling_case_2} as
\[
\frac{\Delta\tilde x}{h_n}   = 2x_{n/2} \beta_n.
\]
If the empirical measures $\mu_n$ in \eqref{measure} converge, then $x_{n/2}  = O(1)$;  this equality therefore indicates that $\beta_n$ is a measure of the aspect ratio $\Delta \tilde x/h_n$, or, put differently, $n\beta_n$ is a measure of the total length of the pileup, relative to $h_n$.

Cases 2--4, therefore, can be understood heuristically as follows:
\begin{itemize}
\item If $\beta_n\to 0$ and $n\beta_n\to\infty$ (case $3$, and Figure~\ref{subfig:case3} below), then the range of the ratio $|\tilde x_i-\tilde x_j|/h_n$, which appears as an argument of $V$ in~\eqref{discreteenergy}, asymptotically covers the whole range from $0$ to $\infty$. In this case the discrete system effectively samples the integral $\int V$, and this integral therefore appears in the limit energy~\eqref{def:limitenergyE3}.
\item If $\beta_n\to c>0$ (case 4 and Figure~\ref{subfig:case4}), then the sampling of $V$ does not refine, but remains discrete, and instead of the integral $\int V$ we find the discrete sampling $V_{\mathrm {eff}}$~\eqref{def:limitenergyE4}.
\item If $n\beta_n\to c>0$ (case 2 and Figure~\ref{subfig:case2}), then the pile-up is not long enough to cover the whole of the integral of $V$. In addition, in this case the length scales of $\mu_n$ and of $V$ are exactly the same, and a convolution integral results.
\end{itemize}

\subfigcapskip23pt
\subcapcentertrue
\begin{figure}[ht]
\centering
\small
\subfigure[$n\beta_n\to c$: total length of the pile-up $n\Delta \tilde x$ remains $O(h_n)$\label{subfig:case2}]{%
\labellist
\pinlabel $\dfrac{n\Delta \tilde x}{h_n}$ [t] at 39 1
\pinlabel $V$ [bl] at 73 27
\endlabellist
\includegraphics[height=3cm]{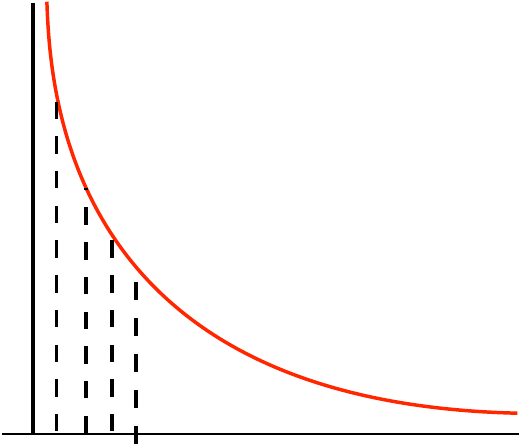}}
\qquad
\subfigure[$1/n\ll\beta_n\ll 1$: the full range $[0,\infty)$ is sampled\label{subfig:case3}]{%
\labellist
\pinlabel $\dfrac{n\Delta \tilde x}{h_n}$ [t] at 121 1
\pinlabel $\dfrac{\Delta \tilde x}{h_n}$ [t] at 21 1
\pinlabel $V$ [bl] at 73 27
\endlabellist
\includegraphics[height=3cm]{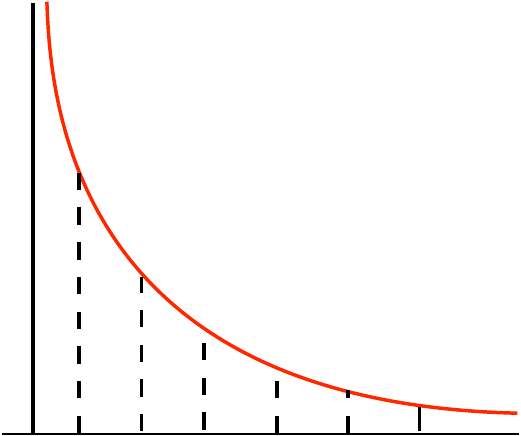}}
\qquad
\subfigure[$\beta_n\to c$: the first dislocation $\tilde x_1\approx \Delta \tilde x$ is of the same order as $h_n$\label{subfig:case4}]{
\labellist
\pinlabel $\dfrac{n\Delta \tilde x}{h_n}$ [t] at 121 1
\pinlabel $\dfrac{\Delta \tilde x}{h_n}$ [t] at 37 1
\pinlabel $V$ [bl] at 73 27
\endlabellist
\includegraphics[height=3cm]{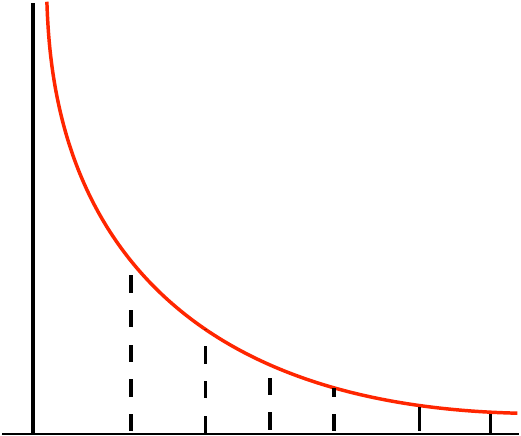}}
\caption{Cases 2--4 of Theorem~\ref{th:main}. Since $\beta_n$ is a measure of the aspect ratio $\Delta\tilde x/h_n$, the scaling of $\beta_n$ determines which values the ratio $(\tilde x_i-\tilde x_j)/h_n$ takes in the argument of $V$ in~\eqref{discreteenergy}.}
\end{figure}
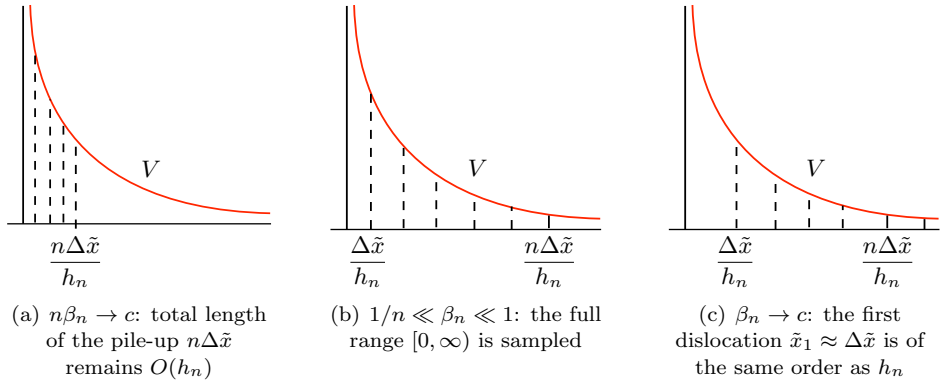

Case 1, where $\beta_n$ is so small that $n\beta_n\to 0$, is a variant of case 2, but now the dislocation walls are pushed completely into the logarithmic singularity of $V$ at the origin (since by~\eqref{rescaling-x:case1} the typical total length of the pile-up is $n K_n/\sigma_n =h_nn^2\beta_n^2$, and this is small with respect to $h_n$). We observe that by the definition of $\beta_n$ \eqref{betan} this situation corresponds to strong forcing, which pushes the dislocation walls closer to each other.
Because of the scaling dependence of the logarithm, a multiplicative rescaling in space (in order to make the sequence $\mu_n$ converge to a non-trivial limit) results in an additive rescaling of $\E$. The corresponding picture is similar to Figure~\ref{subfig:case2}.

In case 5, where $\beta_n$ is large, the value of $(\tilde x_i-\tilde x_j)/h_n$ also becomes large; even the two closest dislocation walls have distance asymptotically larger than $h_n$. Then the dislocations only sample the exponential tail of $V$. By the definition of $\beta_n$ \eqref{betan} this case corresponds to very weak forcing. The degenerate nature of the limit of the interaction energy, which is zero if $\mu=\rho\,dx$ with $\rho\leq 1$, and $+\infty$ otherwise, arises from the `winner takes all' behaviour of the exponential function.

\bigskip

\emph{Other possibilities for $\beta_n$.} One might wonder whether other scaling behaviour of $\beta_n$ could give different results. Although it is certainly possible to construct sequences $\beta_n$ that do not fit into the five classes above, by taking subsequences one can reduce the behaviour to one of these five possibilities. Of course, if different subsequences have different asymptotic behaviour, then one does not expect the functionals to converge; this non-convergence is a further indication that one should separate the cases by dividing into subsequences.

\subsection{Comparison with mesoscopic models in the engineering literature}\label{sec:consequences}

As mentioned above, one motivation for this research is the derivation of a model describing the behaviour of \emph{densities} of dislocations from a more fundamental microscopic model described by the discrete energy~\eqref{defV}. The need for a rigorous derivation of such a dislocation-density model is underlined by the fact that multiple models exist in the literature (see Table~\ref{Table}) that are inconsistent with each other and whose range of validity is not clear.

In the case of this paper, straight parallel edge dislocations in a single slip system, the upscaled evolution equation for the dislocation density (or measure) $\mu$ is expected to be of the form
\begin{equation}\label{mech:1}
\partial_t \mu + \partial_x (v\mu) = 0.
\end{equation}
Here $v(x,t)$ is the velocity of dislocations at $(x,t)$, and is usually taken to be
\begin{equation}\label{mech:2}
v = \frac1B (\sigma_{\mathrm{int}}-\sigma).
\end{equation}
Here $B$ is a mobility coefficient, $\sigma$ is the externally imposed shear stress (as above) and $\sigma_{\mathrm{int}}$ is the shear stress field that the dislocations themselves generate, which is assumed to depend on the dislocation density and on the gradient of the density, i.e., $\sigma_{\textrm{int}} = \sigma_{\textrm{int}}(x,\mu,\partial_x\mu)$. The structure \eqref{mech:1}-\eqref{mech:2} arises naturally from the evolution equations for the discrete system,
\begin{equation}
\label{eq:GFn}
\frac{d}{dt} \tilde x_i = - \frac1B \partial_{\tilde x_i} \E(\tilde x_1,\dots,\tilde x_n)
= -\frac K{Bh}\sum_{\substack{j=0\\j\not=i}}^n V'\Bigl(\frac{\tilde x_i-\tilde x_j}h\Bigr)
 - \frac\sigma B,
\end{equation}
which suggests that $\sigma_{\mathrm{int}}$ should be the upscaled limit of the interaction forces, represented by the sum in~\eqref{eq:GFn}.

The different models proposed in the engineering literature differ in the form of the internal stress $\sigma_{\textrm{int}}$ they suggest, as shown by Table~\ref{Table}, and the arguments leading to the specific choice of $\sigma_{\textrm{int}}$ can not always be rigorously justified.

\medskip

In contrast to these  mostly phenomenologically derived expressions, the convergence results of Theorem~\ref{th:main} offer a rigorous characterization of the limiting internal stress $\sigma_{\textrm{int}}$ in the different cases corresponding to our limit models $k=1,\dots,5$.
The Euler-Lagrange equations of the limit functionals $E^{(k)}$, expressed in terms of the measure $\mu$ or the density $\rho$, are (taking $c=1$ for simplicity):

\begin{align*}
&(k=1)& \quad -\frac1{\pi^2}\partial_x\left(\log \ast\mu\right) + 1 = 0; &\qquad\sigma^{(1)}_{\textrm{int}}  = \frac1{\pi^2}\partial_x\left(\log\ast\mu\right);\\
\bigskip
&(k=2)& \quad \partial_x\left(V\ast\mu\right) + 1 = 0; &\qquad\sigma^{(2)}_{\textrm{int}} = - \partial_x\left(V\ast\mu\right);\\
\bigskip
&(k=3)& \quad \left(\int_{\mathbb{R}} V(t)\,dt\right)\partial_x\rho + 1 = 0;
 &\qquad\sigma^{(3)}_{\textrm{int}} = - \left(\int_{\mathbb{R}} V(t)\,dt\right)\partial_x\rho;\\
\bigskip
&(k=4)& \quad \frac{1}{\rho^3}\, V''_{\textrm{eff}}\left(\frac{1}{\rho}\right)\partial_x\rho -1 = 0;
&\qquad\sigma^{(4)}_{\textrm{int}} = \frac{1}{\rho^3}\, V''_{\textrm{eff}}\left(\frac{1}{\rho}\right)\partial_x\rho.
\end{align*}

We leave out the case $k=5$ since its Euler-Lagrange equation is too degenerate to be  useful.

\medskip

The Euler-Lagrange equation in case $k=1$ coincides with the one derived by Eshelby, Frank, and Nabarro~\cite{EshelbyFrankNabarro51} and Head and Louat~\cite{HeadLouat55} in the case of $n$ dislocations in one slip plane--rather than $n$ dislocation walls. This is consistent with the fact that when $\beta_n\ll 1/n$, the dislocation walls are much closer to each other horizontally than the vertical spacing $h_n$ (see the discussion of $\beta_n$ above), and therefore an approximation by a single-slip-plane setup seems appropriate.

The internal stress $\sigma^{(3)}_{\textrm{int}}$ coincides with the one proposed by Evers, Brekelmans and Geers~\cite{EversBrekelmansGeers04}.
As far as we know, the limiting energies $E^{(k)}$ for $k=2,4$ and the internal stress associated with them have not been mentioned in the engineering literature yet.

\medskip
Groma, Csikor, and Zaiser~\cite{GromaCsikorZaiser03} derived the internal stress $\sigma^{(GCZ)}_{\textrm{int}} = -\partial_x\rho/\rho$ (up to constants) starting from a discrete distribution of dislocations where the horizontal and vertical separation of the dislocations is of the same order.
In our formulation this corresponds to the case $k=4$, $\beta_n\sim 1$. Therefore it is interesting to compare $\sigma^{(GCZ)}_{\textrm{int}}$ with $\sigma_{\mathrm{int}}^{(4)}$. As it turns out, $\sigma^{(GCZ)}_{\textrm{int}}$ can be \emph{formally} obtained from $\sigma_{\textrm{int}}^{(4)}$ by making two approximations. The first  consists in disregarding the interaction between walls that are not nearest neighbours, which is equivalent to replacing the effective potential $V'_{\textrm{eff}}$ with $V'$. The second approximation is to substitute the force $V'(s)$ with its first-order Taylor-Laurent expansion close to zero, namely $- \frac{1}{\pi^2 s}$. Via these two approximations $\sigma_{\mathrm{int}}^{(4)}$ reduces (up to a constant) to $\sigma^{(GCZ)}_{\textrm{int}}$.

Although this derivation can formally be made, it {can not} be made rigorous, since the two approximations are mutually incompatible. In fact,
if Taylor-expanding $V$ is sensible, then the logarithmic singularity of $V$ implies that there is interaction with all neighbours, as is always the case for logarithmic interactions (and as is the case for case $k=1$ above). Therefore neglecting all but nearest neighbours is unjustified.
Moreover, the Taylor expansion of $V'_{\mathrm{eff}}$ and of $V'$ close to zero are quite different since when $s$ is small,\footnote{This follows from the two inequalities (we recall that $V$ is a decreasing function in $(0,\infty)$)
\[
V_{\mathrm{eff}}(s) = \sum_{k=1}^\infty V(ks) \leq\sum_{k=1}^\infty \frac1s \int_{(k-1)s}^{ks} V(t)\, dt =  \frac1s\int_0^\infty V(t)\, dt,
\]
and
\[
V_{\mathrm{eff}}(s) = \sum_{k=1}^\infty V(ks) \geq \sum_{k=1}^\infty \frac1s \int_{ks}^{(k+1)s} V(t)\, dt = \frac1s \int_s^\infty V(t)\, dt.
\]
}
  $V(s)\approx -1/\pi^2\log|s|$, while\begin{equation}
\label{asymp:Veff}
V_{\mathrm{eff}}(s)\approx \dfrac1{2|s|}\int_\R V.
\end{equation}
Therefore truncating to nearest neigbours (replacing $V_{\mathrm{eff}}$ by $V$) and then Taylor-expanding $V'$ results in a large error. We refer to the companion paper \cite{GPPSMech} for further discussions on this point and for a more detailed comparison between the internal stresses $\sigma^{(k)}_{\textrm{int}}$ obtained from our derivation and the models proposed in the engineering literature.

\subsection{Related mathematical work on discrete-to-continuum transitions.}
The model of this paper lies halfway between one and two dimensions. Written as~\eqref{discreteenergy}, it is a one-dimensional system, and an example of the general class of two-point interaction energies. There is a large body of research on this type of energy, which roughly falls into two categories. When the interaction energy is superlinear at infinity, the system models the behaviour of elastic solids, and examples of Gamma-convergence of such functionals are given in~\cite[Th.~1.22]{BraidesMaria06} (see also \cite{AliCic} and \cite{ACP}).
When the functional is bounded at infinity with a global minimum at finite distance, such as in the case of the Lennard-Jones potential $r \mapsto r^{-12}-r^6$ or the Blake-Zisserman potential $r\mapsto \min\{r^2,1\}$~\cite{BlakeZisserman}, such two-point interaction energies lead to models of fracture (see e.g. \cite{BraidesDalMasoGarroni99}, \cite{BraidesCicalese07} and \cite{BLO}).
The functional $V$ in~\eqref{defV} is neither of these, being purely repelling and convex away from the singularity. While the methods that we use are inspired by the general works in this area, we know of no work that deals specifically with this type of functional.

There are various previous works that focus on the behaviour of minimizers rather than on the functional. The early work by Eshelby, Frank and Nabarro~\cite{EshelbyFrankNabarro51} mentioned before studies the case of a single row of dislocations ($h=\infty$) and proves rigorously the asymptotic distribution of the dislocations.
Hall~\cite{Hall11} studies the wall setup, chooses the specific regime $\beta_n\sim n^{-1/2}$, and proves convergence of stationary states using formal methods. Finally we should mention the numerical study~\cite{DeGeusPeerlingsHirschberger11TR} in which the correct asymptotic scaling of the regime $1/n\ll\beta_n\ll 1$ was already found.
Mesarovic and collaborators~\cite{BaskaranMesarovicetc10,MesarovicBaskaranetc10} derive a continuum dislocation model from the discrete wall setup by means of a two-step upscaling: first the dislocations are smeared out in the slip plane and then in the vertical direction. Upscaling in the two directions separately, though, produces a significant error (referred to by the authors as ``the coarsening error") that needs to be corrected by adding an ad hoc term to their continuum model.

\medskip

At the same time, the structure of the walls in Figure~\ref{Wall} is an attempt to make some progress in the  problem of upscaling two-dimensional collections of dislocations. This is a hard problem, and the main difficulty can be recognized as follows. If we consider a field of edge dislocations in two dimensions at points $\{\bx_i\}_{i=1}^N\subset \R^2$, and formulate the corresponding empirical measure on~$\R^2$,
\[
\mu_N := \frac1N \sum_{i=1}^N \delta_{\bx_i}
\]
then the interaction energy for this system is essentially
\[
\iint_{\R^2\times \R^2} V_{\mathrm {edge}}(\bx-\by) \mu_N(d\bx)\mu_N(d\by),
\qquad\text{where}\qquad
V_{\mathrm{edge}}\bigl((x_1,x_2)\bigr) = \frac{x_1^2}{x_1^2+x_2^2}
- \frac12 \log (x_1^2+x_2^2).
\]
The function $V_{\mathrm{edge}}$ is singular at the origin, and therefore a simple weak convergence of $\mu_N$ in the sense of measures to some $\mu$ does not allow us to pass to the limit.

To make things worse, $\partial_{x_1}V_{\mathrm{edge}}$ takes both signs along the line $x_1=\textrm{constant}$.  This indeterminacy causes a phenomenon of cancellation, and surprisingly this cancellation can be \emph{complete}~\cite{RoyPeerlingsGeersKasyanyuk08}: if we consider a continuous vertical line of smeared-out edge dislocations  (i.e. the limit of a wall when $h\to0$), then the total force exerted by this continuous wall on any other edge dislocation \emph{vanishes}~\cite{RoyPeerlingsGeersKasyanyuk08}.  This cancellation is the reason why the tails of $V$ decay exponentially, even though $V_{\mathrm{edge}}$ only decays logarithmically.

Because of the multiple signs of $\partial_{x_1} V_{\mathrm {edge}}$ and this cancellation, also a more advanced argument along the lines of~\cite{SandierSerfaty10TR} does not apply. Indeed, the results of this paper show how the relative spacing in horizontal and vertical directions has a major impact on the limiting energy. This relative spacing, the aspect ratio of the lattice of dislocations, is weakly characterised by $\beta_n$, which we discuss below.

\subsection{Comments}
In this section we collect a number of comments on the discrete model and on the results of this paper.

\bigskip

\emph{Conditions on $V$.} While we perform the calculations in this paper for the exact functional $V$ in~\eqref{defV}, with minor changes the results can be generalized to any function $V$ satisfying
\begin{enumerate}
\item $V:\R\to\R$ is non-negative, even, and convex on $(0,\infty)$;
\item $V$ has a logarithmic singularity at the origin;
\item $V$ has exponential tails.
\end{enumerate}

\bigskip

\emph{On the choice of $\Gamma$-convergence.} Our $\Gamma$-convergence result implies convergence of minimizers, and is stronger in a number of ways. For instance, $\Gamma$-convergence of $E_n^{(k)}$ implies that $E_n^{(k)}+F$ also $\Gamma$-converges whenever $F$ is continuous. This allows us to deduce a similar convergence result, for instance, for a functional of the form
\[
\sum_{i=1}^n\sum_{\stackrel{j=0}{j\neq i}}^n V\left(\frac{\tilde x_i-\tilde x_j}{h}\right) + \sum_{i=1}^n f(\tilde x_i),
\]
for any continuous function $f$, allowing us to consider more general, non-constant forcing terms. If in addition $\liminf_{x\to\infty} f(x)=+\infty$, then a similar compactness result also holds. Note, however, that such a functional obviously behaves differently under rescaling of the $\tilde x_i$.

A second reason why $\Gamma$-convergence is a stronger result is the role that it plays in convergence of the corresponding \emph{evolutionary} problems, i.e. the ordinary differential equations~\eqref{eq:GFn}. That system is a gradient flow, and a method such as in~\cite{SandierSerfaty04} makes use of the $\Gamma$-convergence of $\E$ (and other properties) to pass to the limit $n\to\infty$ in such a system.

\bigskip

\emph{Connection between the limit functionals.} The transitions between the five different limiting functionals of Theorem~\ref{th:main} are continuous. For instance, if in $E^{(2)}$ in~\eqref{def:limitenergyE2} we take the limit $c\to\infty$, then $ s\mapsto cV(cs)$ converges to $(\int V)\delta$, and we recognize the corresponding {single} integral in~\eqref{def:limitenergyE3}.
In the case of $E^{(4)}$, in the limit $c\to0$ we approximate $V_{\mathrm {eff}}(s)$ by its leading order Taylor-Laurent development at the origin, which is $(1/2s)\int_\R V$ by \eqref{asymp:Veff}, upon which $E^{(4)}$ becomes equal to $E^{(3)}$.

Similar transitions exist from $E^{(2)}$ to $E^{(1)}$ in the limit $c\to 0$, and from $E^{(4)}$ to $E^{(5)}$ in the limit $c\to\infty$.

\bigskip

\emph{Boundary layers.} Figure~\ref{DC3} shows a good match over most of the domain, with a sharp boundary layer near the origin. The reason for this boundary layer can be recognized in the fact that $1/n\beta_n\approx 0.08$ is about one order of magnitude smaller than the domain of the density. Such boundary layers are well known in the theory of interacting particles with next-to-nearest neighbours (see e.g.~\cite{BraidesCicalese07}), and we believe that the effect here is similar. Note that in Figure~\ref{DC4} the boundary layer is thinner, and indeed there $1/n\beta_n\approx 0.006$.

\bigskip

\emph{Generalisations.}
Baskaran et al.~\cite{BaskaranMesarovicetc10} and \cite{MesarovicBaskaranetc10} study the same setup with arbitrary angle between the slip planes and the obstacle. They point out that orthogonal slip planes are a  special case among all angles, and an obvious avenue of generalization is to understand the general case. Other generalizations include dislocations of multiple signs, creation and annihilation effects, and convergence of the evolution equations.

\subsection{Organisation of this paper}
In Section~\ref{sec:preliminaries} we prepare the stage for the main proofs, by  introducing equivalent formulations for the rescaled energies and a characterization for the lower-semicontinuity of the limit functionals.
Section~\ref{sec:proof} is devoted to the proofs of the five cases of Theorem~\ref{th:main}, and Theorem~\ref{lem:uniqueness} and Corollary~\ref{cor:minimizers-converge} are proved in Section~\ref{sec:uniqueness}.

\section{Preliminaries}
\label{sec:preliminaries}

In this section we collect a number of preliminary steps leading to the proof of  Theorem~\ref{th:main}. We start with rewriting the discrete functionals in a number of different, equivalent forms. In Section~\ref{Comp:Section} we derive an equivalent characterization of the weak convergence of measures, and in Section~\ref{subsec:lsc} we characterize the lower semicontinuous envelope of functionals of the form $\int f(u')$.

\subsection{Notation}
Here we list some symbols and abbreviations that are going to be used throughout the paper.

\medskip

\noindent
\begin{small}
\begin{tabular}{lll}
$\Omega$        & domain $[0,\infty)$ &\\
$\Lebesgue$ & one-dimensional Lebesgue measure restricted to $\Omega$\\
$\mathcal M(\Omega)$ & non-negative Borel measures on $\Omega$ of mass $1$ &\\
$C_b(A)$ & continuous and bounded functions in $A\subseteq\mathbb{R}$\\
$||\mu||_{TV(A)}$ & total variation of $\mu\in\mathcal M(\Omega)$ in $A\subset \dom$  \\
$\d \nu/\d \mu$ & Radon-Nikodym derivative of $\nu$ with respect to $\mu$\\
$\wmm$ & $\mu_n\otimes\mu_n$ without the diagonal terms (see \eqref{def:wmm})\\
$E_n^{(k)}$, $k\in\{1,2,3,4,5\}$ & discrete energies (see Theorem~\ref{th:main} and Section \ref{subsec:rewriting})\\
$E^{(k)}$, $k\in\{1,2,3,4,5\}$ & limit energies (see Theorem~\ref{th:main})
\end{tabular}
\end{small}

\medskip

Also, we write e.g. $\int_\dom f\, d\mu$ instead of $\int_0^\infty f\, d\mu$, since the latter is ambiguous when~$\mu$ has an atom at zero.

\subsection{Rewriting the functionals}
\label{subsec:rewriting}

The continuum limit functionals $E^{(1)}$ and $E^{(2)}$ are convolution integrals, and this suggests reformulating the corresponding functionals at finite $n$ also as convolution integrals. For given $\mu_n =\frac1n \sum_{i=1}^n \delta_{x_i}$, we define the measure $\wmm$ as the product measure $\mu_n\otimes\mu_n$ without the diagonal:
\begin{equation}
\label{def:wmm}
\wmm(A) := \frac1{n^2}\sum_{i=1}^n\sum_{\substack{j=1\\j\not=i}}^n \delta_{(x_i,x_j)}(A)
\qquad\text{for any Borel set }A\subset\doms.
\end{equation}
Omitting the diagonal does not change the limiting behaviour:
\begin{lem}
\label{lem:boxtimes}
If $\mu_n\weakto \mu$, then $\wmm\weakto \mu\otimes\mu$.
\end{lem}

\begin{proof}
Take $\varphi\in C_b(\doms)$. Then
\[
\int_\doms \varphi \, d\wmm  - \int_\doms \varphi\, d\mu\otimes\mu
= \int_\doms\varphi \, d(\wmm-\mu_n\otimes\mu_n) +
  \int_\doms \varphi \, d(\mu_n\otimes\mu_n - \mu\otimes\mu).
\]
The second term on the right-hand side converges to zero since $\mu_n\weakto \mu$, and the first is bounded by $ \|\varphi\|_\infty/n$ and therefore also converges to zero.
\end{proof}

With this notation we can write $E_n^{(1)}$ in a number of different, equivalent forms:
\begin{align}
E_n^{(1)}( x_1,\dots,x_n) &= \frac1{n^2K_n}\, {\E(\tilde x_1,\dots,\tilde x_n)} + \frac1{2\pi^2} (\log 2\pi n^2\beta_n^2 -1)\notag\\
&= \frac1{2n^2} \sum_{i=1}^n\sum_{\substack{j=1\\j\not=i}}^n \tilde V_n(n^2\beta_n^2 ( x_i- x_j)) + \frac1n \sum_{i=1}^n  x_i\notag\\
&= \frac1{n^2} \sum_{k=1}^n\sum_{j=1}^{n-k} \tilde V_n(n^2\beta_n^2 ( x_{j+k}- x_j)) + \frac1n \sum_{i=1}^n  x_i\notag\\
&= \frac1{2} \int_\doms \tilde V_n(n^2\beta_n^2(x-y))\, \wmm(dxdy) + \int_\dom x\, \mu_n(dx).
\label{rewrite-E1}
\end{align}
Here $\tilde{V}_n(s):= V(s) + \pi^{-2}(\log(2\pi n^2\beta_n^2) - 1)$ is a \textit{renormalized} energy, obtained by removing a \textit{core energy} from the energy density $V$.

Similarly we rewrite
\begin{align}
E_n^{(2)}( x_1,\dots, x_n) &= E_n^{(3)}( x_1,\dots, x_n) = E_n^{(4)}( x_1,\dots, x_n)\notag\\
&= \frac{\beta_n}{nK_n}\; \E(\tilde x_1,\dots,\tilde x_n)\notag\\
&= \frac{\beta_n}{n} \sum_{k=1}^n\sum_{j=1}^{n-k} V(n\beta_n(x_{j+k}-x_j)) + \frac1n \sum_{j=1}^n x_j\notag\\
&= \frac{n\beta_n}2 \iint_\doms V(n\beta_n(x-y))\, \wmm(dxdy) + \int_\dom x\, \mu_n(dx),
\label{rewrite-E24-integral}
\end{align}

and

\begin{align}
E_n^{(5)}( x_1,\dots, x_n) &= \frac{2\pi}{nK_n}\frac{\beta_n^2}{\log\left(\frac{2}{\pi}\beta_n^2\right)}\; \E(\tilde x_1,\dots,\tilde x_n)\notag\\
&= \frac{2\pi\beta_n^2}{n\log\left(\frac{2}{\pi}\beta_n^2\right)} \sum_{k=1}^n\sum_{j=1}^{n-k} V\left(\frac{n}{2\pi}\,\log\left(\frac{2\beta_n^2}{\pi}\right)(x_{j+k}-x_j)\right) + \frac1n \sum_{j=1}^n x_j.
\label{rewrite-E5}
\end{align}


\subsection{Convergence concepts and compactness}\label{Comp:Section}

As already discussed in the introduction, there are two natural ways of describing the positions of a row of dislocation walls:
\begin{enumerate}
\item[(i)] The position $x^n_i$ as a function of particle number $i$. One can make this formulation slightly more useful by reformulating it in terms of increasing functions $\xi^n:[0,1]\to\dom$,  such that $\x^n(i/n) = x^n_i$, with linear interpolation.
\item[(ii)] A measure $\mu_n = \frac1n \sum_{i=0}^n \delta_{x^n_i}$.
\end{enumerate}

In the introduction we mentioned the formulation in terms of measures as the basis for convergence results. However, in the proofs it will sometimes be useful to use the formulation in terms of  functions $\xi^n$. Since we intend the resulting $\Gamma$-convergence to be independent of which formulation we choose,  we choose a single concept of convergence and formulate this equivalently for $\x^n$ and for $\mu_n$. This is the content of the next theorem.

\begin{thm}\label{compactness}
Let $(x_i^n)$ be a sequence of n-tuples such that $x_0^n=0$ and $x_i^n \leq x_{i+1}^n$ for every $n$ and for every $i=0,\dots,n-1$.
Let $\xi^n:(0,1)\to \mathbb{R}_+$ be the affine interpolations of $x_i^n$, i.e.
\begin{equation}\label{pa}
\x^n(s) := x^n_i + n(x^n_{i+1}-x^n_i)\left(s-\frac{i}{n}\right), \quad \mbox{for}\quad s\in \left(\frac{i}{n},\frac{i+1}{n}\right),
\end{equation}
and define the measures $\mu_n \in \mathcal{M}(\dom)$ by
\begin{equation}
\label{def:mu_n}
\mu_n := \frac{1}{n}\sum_{i=1}^n \delta_{x_i^n}.
\end{equation}
Then the following convergence concepts are equivalent:
\begin{enumerate}
\item[(i)] 
$\x^n$ converges to $\x$ in $BV(0,1-\delta)$ for each $0<\delta<1$ (we indicate this as `convergence in $\BVloc(0,1)$');
\item[(ii)] 
$\mu_n $ converges weakly to $\mu$.
\end{enumerate}
If the limit function $\x$ is  strictly increasing and a.e.\ approximately differentiable, then it is related to the limit measure $\mu$ by the formula
\begin{equation}
\label{link:mu-xi}
\mu(dy) = \frac{dy}{\x'(\x^{-1}(y))}.
\end{equation}

Finally, if
\begin{equation}\label{D:L1}
\sup_n \frac1n\sum_{i=1}^n x^n_i < \infty,
\end{equation}
then the sequences $\mu_n$ and $\x^n$ are compact in this topology.
\end{thm}

Note that the weak topology on the space of non-negative Borel measures $\M(\Omega)$ of unit mass is generated by a metric (see e.g.~\cite[Remark~5.1.1]{AmbrosioGigliSavare05} or~\cite[p.~72]{Billingsley99}), and therefore there is no need to distinguish between compactness and sequential compactness.

\begin{proof}
First we prove (i) $\Longrightarrow$ (ii). Let $\overline{\x}^n:[0,1]\to\mathbb{R}_+$ denote the piecewise constant function such that $\overline{\x}^n(s)=x^n_i$ for $s\in \left(\frac{i-1}{n},\frac{i}{n}\right]$, for every $i=1,\dots,n$. We have for sufficiently large $n$,
\[
0\leq \int_0^{1-\delta} (\overline \x^n(s)-\x^n(s))\, ds
\leq \frac1n\sum_{i=0}^{\lceil n(1-\delta)\rceil} (x^n_{i+1}-x^n_i)
\leq \frac1n \|\x^n\|_{TV(0,1-\delta/2)} \stackrel{n\to\infty}\longrightarrow 0,
\]
so that $\overline{\x}^n \to \x$ in $L^1_{\mathrm{loc}}(0,1)$ and, after extracting a subsequence without changing notation, $\overline{\x}^n \to \x$ pointwise a.e.

Let now $\varphi\in C_b(\mathbb{R})$ be a test function (for the weak convergence of measures); then
\begin{align*}
\int_{-\infty}^{\infty}\varphi(y)\,\mu_n(dy)&= \frac1n\sum_{i=0}^n \varphi(x^n_i) =
\frac1n \varphi(0) + \sum_{i=0}^{n-1}\int_{\frac{i}{n}}^{\frac{i+1}{n}}\varphi(\overline{\x}^n(s))\,ds = \frac1n \varphi(0) + \int_0^1\varphi(\overline{\x}^n(s))\,ds,
\end{align*}
and since $\overline{\x}_n\to \x$ a.e.,
\[
\lim_{n\to\infty}\int_0^1\varphi(\overline{\x}^n(s))\,ds = \int_0^1\varphi(\x(s))\,ds.
\]
By the uniqueness of this limit the whole sequence $\mu_n$ converges. By defining $\mu\in \M(\dom)$ through
\begin{equation}
\label{rel:mu-xi}
\forall \varphi\in C_b(\R): \qquad \int_{-\infty}^{\infty}\varphi(y)\,\mu(dy) = \int_0^1\varphi(\x(s))\,ds,
\end{equation}
we have proved that $\mu_n\weakto \mu$.

The identity~\eqref{rel:mu-xi} expresses the property that $\mu$ is the push-forward under $\x$ of the Lebesgue measure $ds$ on $(0,1)$. It follows by~\cite[Lemma~6.5.2]{AmbrosioGigliSavare05} that whenever $\x$ is strictly increasing and a.e. approximately differentiable, then
\[
\mu(dy)= \frac{dy}{\x'(\x^{-1}(y))}.
\]

Next we prove (ii) $\Longrightarrow$ (i). Since $\mu_n$ is assumed to be of the form~\eqref{def:mu_n}, we can construct the positions $x_i^n$ and the linear interpolation $\x^n$ as above. The convergence of $\mu_n$ implies that the sequence $\mu_n$ is tight, which implies in turn that for each $\delta>0$,
\[
\sup_n \sup_{i:i/n\leq 1-\delta} x^n_i <\infty,
\]
and therefore that $\sup_n \x^n(1-\delta)=: M<\infty$.

Therefore, since $\x^n(0)=0$ by \eqref{pa}, we have the bound
\begin{equation}
\int_0^{1-\delta} (1-s) (\x^n)'(s)\,ds =  \delta \x^n(1-\delta) + \int_0^{1-\delta}\x^n(s)\,ds \leq M.
\end{equation}
Therefore, using the monotonicity of $x^n$ we have that
\[
M  \geq \int_0^{1-\delta} (1-s)(\x^n)'(s)\,ds \geq \delta \int_0^{1-\delta}|(\x^n)'(s)|\,ds.
\]
This provides a uniform bound for $(\x^n)'$ in $L^1(0,1-\delta)$, and by integration also a uniform bound on $\x^n$ in $L^1(0,1-\delta)$.
Hence the sequence $(\x^n)$ is equibounded in $W^{1,1}(0,1-\delta)$, and therefore converges in $L^1(0,1-\delta)$ and weakly-\textasteriskcentered\ in $BV(0,1-\delta)$ to
a function $\x\in BV(0,1-\delta)$.

Finally, the compactness of the sequence $\mu_n$ follows from the tightness implied by~\eqref{D:L1} and the estimate
\[
\int_\R |x|\,\mu_n(dx)  = \int_0^1 |\x^n(s)|\, ds \leq  \frac1n \sum_{i=1}^n x^n_i.
\]
\end{proof}

\begin{rem}
Note that the limit function $\x$ introduced in the previous theorem is increasing, since it is the pointwise limit of a sequence of increasing functions.
\end{rem}

\subsection{Lower semicontinuity and relaxation}\label{subsec:lsc}

This section is devoted to a lower semicontinuity result for functionals defined on the space of special
functions with bounded variation. More precisely, the next theorem provides an integral representation for the
relaxed functional in a special case.

\begin{thm}\label{relax}
Let $f:(0,\infty)\to\mathbb{R}$ be a convex and decreasing function such that $\lim_{t\to\infty}f(t)=0$.

Let $F:BV_{loc}(0,1)\to \mathbb{R}\cup\{\infty\}$ be the functional defined as
\begin{equation}
F(u):=\begin{cases}\label{W11}
\displaystyle\int_0^1 f(u')\,dt \quad & \mbox{if }\, u\in W^{1,1}(0,1), u \, \mbox{increasing},\\
+\infty \quad & \mbox{otherwise}.
\end{cases}
\end{equation}
Let $\mathcal{H}$ denote the lower semicontinuous envelope of $F$ (relaxation of $F$) on $BV_{loc}(0,1)$ with respect to the $BV_{\mathrm{loc}}(0,1)$-convergence defined in Theorem~\ref{compactness}.

We introduce the functional $\mathcal{F}:BV_{loc}(0,1)\to \mathbb{R}$ defined, for $u$ increasing, as
\begin{equation}\label{F:BV}
\mathcal{F}(u):= \int_0^1 f(u')\,dt,
\end{equation}
where $u'$ denotes the absolutely continuous part (with respect to the one-dimensional Lebesgue measure) of the measure $Du$, which is the distributional
gradient of $u$. Then we have
$$
\mathcal{F}= \mathcal{H}.
$$
\end{thm}

\begin{proof}
We first note that by construction $\mathcal{F}\leq F$ on $BV_{loc}(0,1)$. Since $\mathcal F$ is lower semicontinuous with respect to strong convergence in $L^1_{\mathrm{loc}}(0,1)$, by e.g.~\cite[Proposition 5.1--Theorem 5.2]{AmbrosioFuscoPallara00},  it follows that
$$
\mathcal{F}\leq\mathcal{H}.
$$

For the opposite inequality we need to show that, for a given $u\in BV_{loc}(0,1)$, $u$ increasing, there exists an approximating sequence
$(u^\ell)\subset W^{1,1}(0,1)$, $u^{\ell}$ increasing for every $\ell$, such that $u^{\ell} \to u$ in $L^1_{\mathrm{loc}}$ and
\begin{equation}\label{rel:limsup}
\limsup_{\ell\to\infty}\int_0^1f\left((u^{\ell})'\right)\,dt \leq \int_0^1f(u')\,dt.
\end{equation}

\medskip

For the construction of the sequence $(u^\ell)$ we proceed as follows. We first approximate the distributional gradient $Du$ of $u$
with $L^1$ functions, say $w^\ell$, with respect to the weak convergence in measure. Then we construct approximations $u^\ell$ as (properly defined)
anti-derivatives of $w^\ell$ and will be therefore in $W^{1,1}$ by construction. This argument is strictly one-dimensional, since it makes use
of the property that every function is a gradient.

\medskip

We now go through the details of the proof.

\medskip

\emph{Step 1: Approximation of $Du$ with $L^1$ functions}. We decompose the distributional gradient $Du$ as $Du = u' + D^su$ into its absolutely continuous part and singular part with respect to the Lebesgue measure. Since $u$ is increasing, both $u'$ and $D^su$ are non-negative measures (being mutually singular).
We notice that the absolutely continuous gradient $u'$ (identified with its density with respect the Lebesgue measure) is by definition a nonnegative $L^1$-function; therefore it is sufficient to approximate the singular measure $D^su$ with nonnegative functions in $L^1$. Let $(g^\ell)$, with $g^\ell\in L^1(0,1)$ be such an approximation and define
\begin{equation}\label{L1:approxDu}
w^\ell:= u' + g^\ell;
\end{equation}
then $w^\ell \in L^1(0,1)$, $w^\ell\geq 0$ a.e. and $w^\ell \rightharpoonup Du$ weakly in measure.

\medskip

\emph{Step 2: Approximation of $u$.} We notice that, for the construction of the approximating sequence, we can assume that $S_u=\emptyset$.
Indeed, let us assume instead that $S_u \neq \emptyset$; by the locality of the argument we are going to use, it is not restrictive
to assume that $S_u=\{t^*\}$.

We define continuous approximations of $u$ as
\begin{equation*}
u_\varepsilon(t):=
\begin{cases}
\bigskip
u(t) &\mbox{if } \, t\in (0,t^*-\varepsilon)\cup(t^* + \varepsilon, 1)\\

u(t^* -\varepsilon) + \frac{u(t^*+\varepsilon) - u(t^*-\varepsilon)}{2\varepsilon}(t-t^* + \varepsilon) &\mbox{if } \, t\in (t^*-\varepsilon,t^* + \varepsilon).
\end{cases}
\end{equation*}
Then, clearly,
\begin{equation*}
u'_\varepsilon(t):=
\begin{cases}
\bigskip
u'(t) &\mbox{if } \, t\in (0,t^*-\varepsilon)\cup(t^* + \varepsilon, 1)\\

\frac{u(t^*+\varepsilon) - u(t^*-\varepsilon)}{2\varepsilon} &\mbox{if } \, t\in (t^*-\varepsilon,t^* + \varepsilon).
\end{cases}
\end{equation*}
For the approximating sequence $u_\varepsilon$ we have
\begin{align*}
\int_0^1f(u'_{\varepsilon}(s))ds
&=\int_{(0,1)\setminus(t^*-\varepsilon,t^* + \varepsilon)}f(u'(s))ds +
\int_{t^*-\varepsilon}^{t^* + \varepsilon}f\left(\frac{u(t^*+\varepsilon)- u(t^*-\varepsilon)}{2\varepsilon}\right)ds\\
&=\int_{(0,1)\setminus(t^*-\varepsilon,t^* + \varepsilon)}f(u'(s))ds + 2\varepsilon f\left(\frac{u(t^*+\varepsilon)- u(t^*-\varepsilon)}{2\varepsilon}\right)\\
&\leq \int_0^1 f(u'(s))ds +
2\varepsilon f\left(\frac{u(t^*+\varepsilon)- u(t^*-\varepsilon)}{2\varepsilon}\right).
\end{align*}
The decay at infinity of $f$ implies therefore that
$$
\limsup_{\varepsilon\to 0} \int_0^1f(u'_{\varepsilon}(s))ds \leq \int_0^1 f(u'(s))ds.
$$
Therefore we can assume that $u$ is continuous.
\bigskip

We define the primitive of the function $w^\ell$ defined in \eqref{L1:approxDu} as
$$
u^\ell(t):= u(0)+\int_0^tw^\ell(s)ds.
$$
It follows that $u^\ell \in W^{1,1}(0,1)$, $u^\ell(0)=u(0)$ and $(u^\ell)' = w^\ell$, which converges weakly to $Du$ in measure.

Since $(u^\ell)$ is bounded in $W^{1,1}$, then it converges weakly in $BV$ to a function $v\in BV(0,1)$. By the weak convergence of $(u^\ell)'$ in measure
it follows that $Du = Dv$ and therefore, since $v(0)=u(0)$, that $u=v$.
Hence, we have constructed a sequence $(u^\ell)\subset W^{1,1}(0,1)$ such that $u^\ell\to u$ in $BV(0,1)$, and hence in $BV_{\mathrm{loc}}(0,1)$.

\medskip

\emph{Step 3: Upper bound for the energies.} Since $(u^\ell)' = u'+ g^\ell$ and $g^\ell\geq 0$ we have by construction that
$$
\int_0^1f((u^\ell)')ds \leq \int_0^1 f(u')ds
$$
for every $\ell$, since $f$ is a decreasing function. The bound \eqref{rel:limsup} follows immediately.
\end{proof}

\section{Proof of Theorem~\ref{th:main}}\label{sec:proof}

We separate Theorem~\ref{th:main} into the five different cases, and state and prove each case separately.

\begin{thm}[Case 2, first critical regime: $\beta_n\sim 1/n$]
\label{th:first-critical}
Let $c_n := n\beta_n \to c>0$ as $n\to \infty$. For this case the functional $E_n^{(2)}$ in~\eqref{rewrite-E24-integral}, which can be rewritten as
\[
E_n^{(2)}(\mu_n) = \frac {c_n}2\iint_\doms V(c_n(x-y))\, \wmm(dxdy) + \int_\dom x\, \mu_n(dx),
\]
$\Gamma$-converges with respect to the weak convergence in measure to the functional $E^{(2)}$ defined for $\mu \in \mathcal{M}(\dom)$ as
\begin{equation}\label{E2c}
E^{(2)}(\mu):= \frac c2\iint_\doms V(c(x-y))\,\mu(dx)\mu(dy) + \int_\dom x\mu(dx).
\end{equation}
In addition, if $E_n^{(2)}(\mu_n)$ is bounded, then $\mu_n$ is weakly compact.
\end{thm}

\begin{proof}
The compactness statement is a direct consequence of Theorem~\ref{compactness}, since the interaction potential $V$ is non-negative.
The remainder of the theorem we first prove under the assumption that $c_n=1$.

\medskip

\textbf{Liminf inequality}.
Let $\mu \in \mathcal{M}(\dom)$ and let $\mu_n$ be a sequence of measures of the form $\mu_n=\frac1n\sum_{i=1}^n\delta_{x^n_i}$ such that $\mu_n{\rightharpoonup} \mu$ weakly in measure.
Since $V\geq0$ is lower semicontinuous on $\R^2$ and $\wmm\weakto \mu\otimes\mu$ by Lemma~\ref{lem:boxtimes},
\[
\liminf_{n\to\infty} \frac12\iint_\doms  V(x-y) \,\wmm(dxdy) \geq \frac12 \iint_\doms V(x-y) \,\mu(dx)\mu(dy).
\]
For the second term we have a similar bound, and therefore
\[
\liminf_{n\to\infty} E_n^{(2)}(\mu_n) \geq E^{(2)}(\mu).
\]

\medskip

\textbf{Limsup inequality}.
It is sufficient to prove the limsup inequality only for a dense class,
\[
A := \Bigl\{\mu\in \M(\dom): \supp \mu \text{ bounded, }\mu \ll \mathcal{L}, \text{ and }
\frac{\d\mu}{\d \mathcal{L}} \in L^\infty\Bigr\}.
\]
This set is dense in $\M(\dom)$, and for any $\mu\in\M(\dom)$ with $E^{(2)}(\mu)<\infty$ an approximating sequence $(\mu_k)\subset A$ can be found such that $\mu_k\weakto\mu$ and $E^{(2)}(\mu_k)\to E^{(2)}(\mu)$. This can be seen, for instance, by defining
\[
\mu_k(dx) = \rho_k(x)\,dx \qquad\text{with }\rho_k(x) = k\mu([x,x+1/k)).
\]
Then by Fubini,
\begin{align*}
\iint_\doms V(x-y) \,\mu_k(dx)\mu_k(dy)
&= k^2\iint_\doms V(x-y) \int_x^{x+1/k}\mu(d\xi) \int_y^{y+1/k}\mu(d\eta)\,dx dy\\
&= \iint_\doms \int_{(\xi-1/k)_+}^\xi \int_{(\eta-1/k)_+}^\eta k^2 V(x-y) \,dy dx\, \mu(d\xi)\mu(d\eta).
\end{align*}
Now one recognizes in the inner two integrals the convolution of the function
\[
(x,y) \mapsto V(x-y)\chi_\dom(x)\chi_\dom(y)
\]
with the characteristic function of the square $[0,1/k)^2$, so that the expression above converges to
\[
\iint_\doms V(x-y)\,\mu(dx)\mu(dy).
\]
This shows that it is sufficient to prove the limsup inequality for all $\mu\in A$.

Take such a measure $\mu\in A$ with Lebesgue density $\rho\in L^\infty(\dom)$, and construct an approximation $\mu_n = \frac1n \sum_{i=1}^n \delta_{x^n_i}$ by defining the points $x_i$ by
\[
\int_0^{x^n_i} \rho(x)\, dx = \frac in.
\]
Then
\[
|x^n_{i+1}-x^n_i|\geq \frac1{n\|\rho\|_\infty}.
\]
Since $\supp \rho$ is bounded, all $x_i^n$ are uniformly bounded, and
\[
\lim_{n\to\infty} \int_\dom x\,\mu_n(dx) = \int_\dom x\, \mu(dx).
\]

Turning to the convolution term, for fixed $m>0$ we write
\[
\frac12\iint_\doms V\,\wmm = \frac12\iint_\doms (V\wedge m) \,\wmm + \frac12\iint_\doms (V-(V\wedge m))\,\wmm.
\]
In the first term the function $V\wedge m$ is bounded and continuous, and this term therefore converges to
\[
\frac12 \iint_\doms (V\wedge m)\, \mu\otimes\mu \leq \frac12 \iint_\doms V\, \mu\otimes\mu.
\]
If $X_m>0$ solves $V(X_m) = m$, then we estimate the second term by
\begin{align*}
\frac12\iint_\doms (V-(V\wedge m))\,\wmm &\leq \frac12 \iint_{\{|x-y|<X_m\}} V\, \wmm\\
&= \frac1{n^2} \sum_{k=1}^n \sum_{j=1}^{n-k} V(x^n_{j+k}-x^n_j) \1{|x^n_{j+k}-x^n_j|< X_m}\\
&\leq \frac1{n^2} \sum_{k=1}^{\lfloor X_m n \|\rho\|_\infty\rfloor}\sum_{j=1}^{n-k} V(x^n_{j+k}-x^n_j)\\
&\leq \frac1{n} \sum_{k=1}^{\lfloor X_m n \|\rho\|_\infty\rfloor} V\Bigl(\frac k{n\|\rho\|_\infty}\Bigr)\\
&\leq \|\rho\|_\infty \sum_{k=1}^{\lfloor X_m n \|\rho\|_\infty\rfloor} \int_{\frac{k-1}{n\|\rho\|_\infty}}^\frac{k}{n\|\rho\|_\infty}V(s)\, ds\\
&\leq \|\rho\|_\infty \int_0^{X_m} V(s)\, ds.
\end{align*}
Therefore
\[
\limsup_{n\to\infty} E_n^{(2)}(\mu_n) \leq E^{(2)}(\mu) + \|\rho\|_\infty \int_0^{X_m } V(s)\, ds.
\]
Since $m>0$ is arbitrary, and since $\lim_{m\to\infty} X_m = 0$, this proves the limsup estimate
\begin{equation}
\label{th:first-critical:limsupaim}
\limsup_{n\to\infty} E_n^{(2)}(\mu_n) \leq E^{(2)}(\mu) .
\end{equation}
\medskip

In order to allow for $c_n\not=1$, we define the scaled measure
\[
\widetilde\mu_n := \frac1n \sum_{i=1}^n \delta_{c_nx_i},
\]
with which
\[
E_n^{(2)}(\mu_n) = \frac {c_n}2\iint_\doms V(x-y)\, \widetilde\mu_n\boxtimes\widetilde\mu_n(dxdy) + \frac1{c_n}\int_\dom x\, \widetilde\mu_n(dx).
\]
The two prefactors in this expression do not change the arguments above, and upon back-transformation the result of the theorem is found.
\end{proof}

\begin{thm}[Case 1, subcritical regime: $\beta_n\ll1/n$]
\label{th:subcritical}
Let $\beta_n>0$ be a sequence such that $n\beta_n\to 0$ as $n\to \infty$. Then the functionals $E_n^{(1)}$ defined in \eqref{rewrite-E1} $\Gamma$-converge to the functional $E^{(1)}$ defined on measures $\mu\in \mathcal{M}(\dom)$ as
\begin{equation}\label{unoC}
E^{(1)}(\mu):= -\frac{1}{2\pi^2}\iint_\doms \log|x-y|\, \mu(dy)\mu(dx)  + \int_\dom x \mu(dx).
\end{equation}
In addition, if $E_n^{(1)}(\mu_n)$ is bounded, then $\mu_n$ is weakly compact.
\end{thm}

\begin{proof}
\textbf{Compactness for the measures $\mu_n$.}
This is the only one of the five cases in which the compactness is non-trivial, since $\tilde V_n$ takes both signs; therefore a bound on $E_n^{(1)}$ does not translate directly into a bound on the second term $\frac1n \sum x_i$. However, by combining the first two terms, such a bound can be obtained, as we now show.

First we show that
\begin{equation}
\label{ineq:bound-V-from-below}
V(t) \geq \hat V(t) := \frac{1-\log2\pi |t|}{\pi^2}\qquad\text{for all $t\not=0$}.
\end{equation}
This follows by remarking that for $t>0$
\begin{equation}
\label{ineq:bound-V-from-below-part1}
V'(t) - \hat V'(t) = -\frac t{\sinh^2\pi t} + \frac1{\pi^2 t}\geq 0,
\end{equation}
and using the expression
\[
V(t) = \frac {2t}{\pi(e^{2\pi t}-1)} - \frac1{\pi^2} \log(1-e^{-2\pi t})
\]
we compute that for $t>0$
\begin{equation}
\label{ineq:bound-V-from-below-part2}
\lim_{t\downarrow 0} V(t) - \hat V(t)
= \lim_{t\downarrow 0}
  \Bigl[\frac {2t}{\pi(e^{2\pi t}-1)} + \frac1{\pi^2}\Bigl\{- \log(1-e^{-2\pi t})
  - 1 + \log 2\pi t\Bigr\}\Big] = 0.
\end{equation}
From~\eqref{ineq:bound-V-from-below-part1} and~\eqref{ineq:bound-V-from-below-part2} we deduce~\eqref{ineq:bound-V-from-below}.

Therefore  the renormalised interaction energy $\tilde{V}_n$ satisfies
\begin{align}
\tilde{V}_n(n^2\beta_n^2 t) &= V(n^2\beta_n^2 t) + \frac{\log(2\pi n^2\beta_n^2) - 1}{\pi^2} \notag\\
&\geq \frac{1-\log(2\pi n^2\beta_n^2 |t|)}{\pi^2} + \frac{\log(2\pi n^2\beta_n^2) - 1}{\pi^2}\notag\\
& = -\frac{1}{\pi^2}\log |t|.
\label{ineq:tildeVn}
\end{align}

Note that for all $t\not=0$
\[
\tilde V_n(n^2\beta_n^2 t) + \frac12 |t| \geq \frac12|t| - \frac1{\pi^2} \log |t| \geq
\frac1{\pi^2}\Bigl(1-\log\frac2{\pi^2}\Bigr)\geq0.
\]
Let $\mu_n$ be a sequence of measures of the form $\mu_n=\frac1n\sum_{i=1}^n\delta_{x_i^n}$ such that $E_n^{(1)}(\mu_n)$ is bounded. We now estimate
\begin{align*}
E_n^{(1)}(\mu_n) &=
\frac1{n^2} \left\{ \frac12\sum_{i=1}^n\sum_{\substack{j=1\\j\not=i}} V_n(n^2\beta_n^2(x_i-x_j)) + \frac14 \sum_{i=1}^n\sum_{j=1}^n (x_i+x_j) \right\}
+ \frac1{2n} \sum_{i=1}^n x_i\\
&\geq \frac1{2n^2} \sum_{i=1}^n\sum_{\substack{j=1\\j\not=i}}\Bigl[ V_n(n^2\beta_n^2(x_i-x_j)) + \frac12 |x_i-x_j| \Bigr]
+ \frac1{2n} \sum_{i=1}^n x_i\\
&\geq \frac1{2n} \sum_{i=1}^n x_i.
\end{align*}
The boundedness of $E_n^{(1)}(\mu_n)$ and Theorem~\ref{compactness} then provide compactness of the sequence $\mu_n$.

\medskip

\textbf{Liminf Inequality.}
Let now $\mu_n$ be a sequence of measures of the form $\mu_n=\frac1n\sum_{i=1}^n\delta_{x_i^n}$ that converges weakly to $\mu$, and note that by Lemma~\ref{lem:boxtimes}, $\wmm\weakto \mu\otimes\mu$.
By \eqref{ineq:tildeVn} we have the bound
\begin{align}
E_n^{(1)}(\mu_n) &\geq -\frac{1}{2\pi^2}\iint_\doms\log|x-y|\, \wmm(dxdy) + \int_\dom x \,\mu_n(dx)\notag\\
&= \iint_\doms\Bigl[-\frac{1}{2\pi^2}\log|x-y| + \frac12 (x+y)\Bigr]\, \wmm(dxdy) + \frac1n\int_\dom x \,\mu_n(dx).
\label{ineq:th:k=1:liminf}
\end{align}
The function between brackets is lower semicontinuous, and by a similar argument as we used for the compactness above it is also bounded from below. Therefore the right-hand side in~\eqref{ineq:th:k=1:liminf} is lower semicontinuous with respect to weak measure convergence, and therefore
\begin{align*}
\liminf_{n\to\infty} E_n^{(1)}(\mu_n) &\geq -\frac{1}{2\pi^2}\iint_\doms \log|x-y|\, \mu(dy)\mu(dx)  + \int_\dom x\, \mu(dx)\\
&= E^{(1)}(\mu).
\end{align*}

\medskip

\textbf{Limsup inequality}.
For the construction of a recovery sequence we first prove a second inequality on $V$ for $t>0$:
\begin{align*}
V(t) &= \frac {2t}{\pi(e^{2\pi t}-1)} - \frac1{\pi^2} \log(1-e^{-2\pi t})\\
&\leq \frac1{\pi^2} - \frac1{\pi^2} \Bigl[-2\pi t + \log(e^{2\pi t}-1)\Bigr]\\
&\leq \frac1{\pi^2} \Bigl[ 1+2\pi t - \log 2\pi t\Bigr],
\end{align*}
from which follows the estimate for all $t\not =0$,
\begin{equation}
\label{est:tildeV}
\tilde{V}_n(n^2\beta_n^2 t)
=  V(n^2\beta_n^2 t) + \frac{\log(2\pi n^2\beta_n^2) - 1}{\pi^2}
\leq \frac1{\pi^2}(2\pi n^2\beta_n^2 - \log |t|).
\end{equation}

The remainder of the argument follows largely the proof of Theorem~\ref{th:first-critical}. Given a similar limit measure $\mu$ and approximating sequence $\mu_n$, we estimate
\begin{align*}
E_n^{(1)}(\mu_n) &= \frac12 \iint_\doms V(n^2\beta_n^2(x-y))\,\wmm(dxdy) + \int_\dom x\, \mu_n(dx)\\
&\leq - \frac1{2\pi^2} \iint_\doms \log|x-y|\,\wmm(dxdy) + \int_\dom x\, \mu_n(dx)
+ \frac{n^2\beta_n^2}\pi.
\end{align*}
Decomposing $-\log|x-y|$ into a part that is bounded and a remainder, as in the proof of Theorem~\ref{th:first-critical}, and repeating the corresponding estimate, one can show that the right-hand side converges to $E^{(1)}(\mu)$. This proves
\[
\limsup_{n\to\infty} E_n^{(1)}(\mu_n) \leq E^{(1)}(\mu).
\]
\end{proof}

\begin{thm}[Case 3, intermediate regime: $1/n\ll\beta_n\ll1$]\label{GE3}
Let $\beta_n>0$ be a sequence such that $\beta_n\to 0$ and $n\beta_n \to \infty$, as $n\to \infty$. Then the functionals $E_n^{(3)}$ defined in \eqref{rewrite-E24-integral} $\Gamma$-converge with respect to weak measure convergence to the functional $E^{(3)}$ defined on measures $\mu\in \mathcal{M}(\Omega)$ as
\begin{equation}
E^{(3)}(\mu) := \begin{cases}
\displaystyle \frac 12 \left(\int_\R V\right) \int_\dom\rho^2(x)\,dx
+ \int_\dom x\rho(x)\, dx & \text{if } \mu = \rho dx,\\
+\infty &\text{otherwise,}
\end{cases}
\end{equation}
which is the same as
\begin{equation}\label{E3c}
E^{(3)}(\x):=\frac12 \left(\int_\R V\right) \int_0^1 \frac{1}{\x'(s)}ds + \int_0^1 \x(s)\,ds,
\end{equation}
when written in terms of $\x\in BV_{loc}(0,1)$, $\x$ increasing, and $\mu$ and $\xi$ are linked by~\eqref{link:mu-xi}.
In addition, if $E_n^{(3)}(\mu_n)$ is bounded, then $\mu_n$ is weakly compact.
\end{thm}

\begin{proof}
Again the compactness statement follows from Theorem~\ref{compactness}.

For the \textbf{liminf inequality} we will make use of the expression \eqref{rewrite-E24-integral} for the energy, i.e.,
\begin{equation}\label{E3m}
E^{(3)}_n(\mu_n)= \frac12\,n\beta_n\iint_\doms V(n\beta_n(x-y))\,\wmm(dxdy) + \int_\dom x \,\mu_n(dx).
\end{equation}
We will prove that for any sequence $\mu_n\weakto \mu$,
\begin{equation}
\label{ineq:th:k=3:liminf}
\liminf_{n\to \infty} E^{(3)}_n(\mu_n) \geq   E^{(3)}(\mu).
\end{equation}
\medskip

Take a sequence $\mu_n\weakto \mu$ such that $E^{(3)}(\mu_n)$ remains bounded. Since the second term
of $E^{(3)}_n(\mu_n)$ is bounded, Theorem~\ref{compactness} guarantees that there exists a measure $\mu$ such that $\mu_n \rightharpoonup \mu$ in measure, at least along a subsequence, and we switch to that subsequence without changing notation. The support of the limit measure $\mu$ lies in $\Omega = [0,\infty)$ by the definition of the extended measures~$\mu_n$.
We split the rest of the proof into three steps.

\medskip

\emph{Step 1: Rewriting the energy in terms of convolutions.}
Let $V_n(t):=n\beta_nV(n\beta_n t)$; we claim that $V_n$ converges to $(\int_{\mathbb{R}} V)\delta_0$ in distributions. Indeed, let
$\psi \in C_0^\infty(\mathbb{R})$; then
\begin{equation*}
\lim_{n\to\infty}\int_{\mathbb{R}}V_n(s)\psi(s)ds = \lim_{n\to\infty}\int_{\mathbb{R}} V(t)\psi\left(\frac t{n\beta_n}\right)dt = \left(\int_{\mathbb{R}}V(t)dt\right) \psi(0),
\end{equation*}
which proves the claim.

\medskip

Now we use the fact that $V = W\ast W$, where $W = \check{U}$ and $U = \sqrt{\hat{V}}$, as proved in the Appendix (Subsection \ref{convolution}). In addition to $W$ we will also use its truncation $W^m := \min\{W,m\}$ for any fixed $m>0$.

Let $W_n$ be defined as $W_n:= \check{U}_n$, with $U_n = \sqrt{\hat{V_n}}$. By the scaling properties of the Fourier transform it follows that  $W_n(t) = n\beta_n W(n\beta_n t)$. Similarly we define $W^m_n(t) := n\beta_n W^m(n\beta_n t)$.
We note that, like in the case of $V_n$, the distributional limits of $W_n$ and $W^m_n$ are $\left(\int_{\mathbb{R}}W\right) \delta_0$ and $(\int_\R W^m)\,\delta_0$. Moreover,
\begin{equation}\label{intVW}
\int_{\mathbb{R}}W(t)\,dt = \hat{W}(0) = \sqrt{\hat{V}}(0) = \sqrt{\int_{\mathbb{R}}V(t)\,dt}.
\end{equation}

\medskip

Note that since $V=W\ast W$,
\[
V_n(x-y) = \int_\R W_n(z-x)W_n(z-y)\, dz,
\]
and therefore
\begin{align*}
\frac12\iint_{\doms} V_n(x-y)\wmm(dxdy) &= \frac1{n^2} \sum_{k=1}^n\sum_{j=1}^{n-k} V_n(x_{j+k}-x_j)\\
&= \frac1{n^2} \sum_{k=1}^n\sum_{j=1}^{n-k} \int_\R W_n(z-x_{j+k})W_n(z-x_j)\, dz.
\end{align*}
We then estimate
\begin{align}
\frac12\int_\R (W^m_n\ast \mu_n)^2
&= \frac1{n^2} \sum_{k=1}^n\sum_{j=1}^{n-k} \int_\R W^m_n(z-x_{j+k})W^m_n(z-x_j)\, dz + \frac1{2n^2} \sum_{j=1}^n \int W^m_n(z-x_j)^2\, dz\notag\\
&\leq \frac1{n^2} \sum_{k=1}^n\sum_{j=1}^{n-k} \int_\R W_n(z-x_{j+k})W_n(z-x_j)\, dz
+ \frac1{2n} \|W^m_n\|^2_2\notag\\
&= \frac12\iint_{\doms} V_n(x-y)\wmm(dxdy) + \frac{\beta_n}2\|W^m\|^2_2
\label{th:intermediate:est:liminf:wmn}\\
&\leq C.
\notag
\end{align}

Therefore we obtain weak convergence in $L^2(\R)$ along a subsequence of $W^m_n\ast \mu_n$ to some $f\in L^2(\R)$.

\medskip

\emph{Step 2: Identification of $f$.}
In order to find the relation between $f$ and $\mu$ we compute the distributional limit of the sequence $W^m_n\ast \mu_n$. Let $\psi\in C^\infty_0(\mathbb{R})$ be a test function; then we have
\begin{align}\label{identifylimit}
\lim_{n\to \infty}\int_{\mathbb{R}} (W^m_n\ast \mu_n)(x)\psi(x)\,dx = \lim_{n\to \infty}\int_{\mathbb{R}}(W^m_n\ast \psi)(x)\mu_n(dx). 
\end{align}
Note that $W^m_n\ast \psi \stackrel{n\to\infty}\longrightarrow \left(\int_{\mathbb{R}}W^m\right)\psi$ uniformly, since $W^m_n\ast \psi$ converges to $(\int W^m)\,\psi$ strongly in $H^1(\R)$. The uniform convergence of $W^m_n\ast\psi$, together with the weak convergence of $\mu_n$ to $\mu$, guarantee that the limit in \eqref{identifylimit} is:
\begin{equation}
\lim_{n\to \infty}\int_{\mathbb{R}} (W^m_n\ast \mu_n)(x)\psi(x)\,dx = \lim_{n\to \infty}\int_{\mathbb{R}}(W^m_n\ast \psi)(x)\mu_n(dx) = \left(\int_{\mathbb{R}}W^m\right) \int_{\mathbb{R}}\psi(x)\mu(dx).
\end{equation}
Therefore, by the uniqueness of the limit of $(W^m_n\ast \mu_n)$ we deduce that $f = \left(\int_{\mathbb{R}}W^m\right) \mu$.
Hence, $\mu$ is absolutely continuous with respect to the Lebesgue measure with a density in $L^2(\mathbb{R})$, i.e., there
exists $\rho \in L^2(\mathbb{R})$, $\rho \geq 0$ a.e., such that $\mu = \rho\, dx$.

\medskip

\emph{Step 3: Lower bound.} From~\eqref{th:intermediate:est:liminf:wmn} it follows that
\[
\liminf_{n\to \infty} \frac 12 \iint_\doms V_n(x-y)\wmm(dxdy)
\geq \frac 12 \left(\int_{\mathbb{R}}W^m\right)^2\int_\dom\rho^2(x)\,dx.
\]
Taking on both sides the supremum over $m>0$ we find
\begin{align}
\liminf_{n\to \infty} \frac 12 \iint_\doms V_n(x-y)\wmm(dxdy)&
\geq \frac 12 \left(\int_{\mathbb{R}}W\right)^2\int_\dom\rho^2(x)\,dx \nonumber\\
&= \frac 12 \left(\int_{\mathbb{R}} V\right) \int_\dom\rho^2(x),dx,
\end{align}
where in the last equality we used the relation \eqref{intVW}. For the second term of the energy we have that, for every $M>0$,
\begin{equation}\label{3:1}
\liminf_{n\to \infty} \int_{\mathbb{R}} x \,\mu_n(dx) \geq \liminf_{n\to \infty}\int_{[0,M]} x \,\mu_n(dx) = \int_0^M x\rho(x)\, dx,
\end{equation}
so that, taking the supremum on all $M>0$
\begin{equation}\label{3:2}
\liminf_{n\to \infty}\int_{\mathbb{R}} x \,\mu_n(dx) \geq \int_\dom x\rho(x) \,dx.
\end{equation}

In conclusion, from \eqref{3:1} and \eqref{3:2} follows the inequality
\begin{align*}
\liminf_{n\to \infty} E^{(3)}_n(\mu_n) \geq  \frac 12 \left(\int_\R V\right) \int_\dom\rho^2(x)\,dx +
\int_\dom x \rho(x)\,dx = E^{(3)}(\mu),
\end{align*}
which is~\eqref{ineq:th:k=3:liminf}.

\medskip
We now continue with the proof of the \textbf{limsup inequality}: for each $\xi\in BV_{\mathrm{loc}}(0,1)$, there exists a sequence $(x^n)_n$ of $n$-vectors $(x^n_1,\dots,x^n_n)$ such that
\begin{equation}
\label{ineq:th:k=3:limsup}
\limsup_{n\to\infty}E^{(3)}_n(x^n_1,\dots,x^n_n) \leq  E^{(3)}(\xi) = \frac12\left(\int_\R V\right) \int_0^1 \frac{1}{\x'(s)}\,ds + \int_0^1 \x(s)\,ds.
\end{equation}
By  Theorem \ref{relax} we can assume without loss of generality that $\x\in W^{1,1}(0,1)$.

\medskip

So, let $\x\in W^{1,1}(0,1)$ be an increasing function such that $E^{(3)}(\x)<\infty$. We can assume that there exists $\varepsilon>0$ such that $\x'\geq \varepsilon$ uniformly on $(0,1)$. Indeed, we can otherwise approximate $\x$ by the sequence $\x_\varepsilon(t):= \x(t) + \varepsilon t$. Clearly $\x_\varepsilon \to \x$ in $W^{1,1}$ as $\varepsilon\to 0$; moreover for the absolutely continuous part of the distributional gradient of $\x_\varepsilon$ we have that $\x'_\varepsilon = \x' + \varepsilon$. Hence $\x'_{\varepsilon}\geq \varepsilon$, since $\x$ is increasing. Also, since $\x'_\varepsilon \geq \x'$ and $f(t)=1/t$ is decreasing we have
\[
\int_0^1 \frac{1}{\x'_\varepsilon(t)}\,dt \leq \int_0^1 \frac{1}{\x'(t)}\,dt,
\]
so that
\[
\limsup_{\varepsilon\to 0} E^{(3)}(\x_\varepsilon) \leq E^{(3)}(\x).
\]
Therefore, from now on we can assume that $\x'\geq \varepsilon$ for some $\varepsilon >0$.

\medskip

For every $n\in \mathbb{N}$ we define the piecewise affine function $\x^n$ and the points $x^n_i$ by $\x^n\left(\frac in\right):= x^n_i := \x\left(\frac in\right)$. Clearly the sequence $\x^n$ converges to $\x$ strongly in $W^{1,1}$. We consider the energy for this sequence,
\begin{align*}
E^{(3)}_n(x^n_1,\dots, x^n_n) &=\frac{\beta_n}{n}\sum_{k=1}^n\sum_{i=0}^{n-k} V\left(\beta_n n(x^n_{i+k}-x^n_i)\right) + \frac{1}{n} \sum_{i=0}^n x^n_i.
\end{align*}
As the second term of the functional is the Riemann sum of the integral of $\x^n$ in $(0,1)$, we focus on the first term.
By the convexity of the energy density $V$ we have that, using Jensen's inequality,
\begin{align*}
\frac{1}{n}\, V\left(\beta_n n(x^n_{i+k}- x^n_i)\right)
= \frac{1}{n}\,V\left(\beta_n k \frac{n}{k}\int_{\frac in}^{\frac{i+k}{n}}(\x^n)'(s)\,ds\right)
\leq \frac{1}{k}\int_{\frac in}^{\frac{i+k}{n}} V\left(\beta_n k (\x^n)'(s)\right)\, ds
\end{align*}
for every $k=1,\dots,n$ and every $i=0,\dots,n-k$. Since
\[
\sum_{i=0}^{n-k}\frac{1}{k}\int_{\frac in}^{\frac{i+k}{n}}V\left(\beta_n k (\x^n)'(s)\right) ds \leq \int_0^1V\left(\beta_n k (\x^n)'(s)\right)\,ds\quad \mbox{for every }k=1,\dots,n,
\]
we have the following estimate for the first term of the energy:

\begin{align}\label{subbfa}
\frac{\beta_n}{n}\sum_{k=1}^n\sum_{i=0}^{n-k} V\left(\beta_n n(x^n_{i+k}-x^n_i)\right) &\leq \beta_n\sum_{k=1}^n \int_0^1 V\left(\beta_n k (\x^n)'(s)\right) ds\nonumber\\
&= \int_0^1 \frac{1}{(\x^n)'(s)}\beta_n (\x^n)'(s)\sum_{k=1}^n V\left(k(\beta_n (\x^n)'(s))\right) ds.
\end{align}

We define $\delta_n(s):= \beta_n(\x^n)'(s)$. Since by assumption $(\x^n)'\geq \varepsilon$ a.e.\ in $(0,1)$, it follows that $n\delta_n(s)\to \infty$ for a.e.\ $s \in (0,1)$. Moreover, since $\x'$ is finite for a.e.\ $s$, $\delta_n(s)\to 0$. It follows that for a.e.\ $s$, the expression
\[
\delta_n(s)\sum_{k=1}^n V\left(k\delta_n(s)\right)
\]
is a Riemann sum for the integral $\int_0^{\infty}V(t)dt = (1/2) \int_\R V$. Therefore letting $n\to \infty$ and by virtue of \eqref{subbfa}, we have the following bound for the energies $E^{(3)}_n$:
\[
\limsup_{n\to\infty}E^{(3)}_n(x^n_1,\dots,x^n_n) \leq \frac12\left(\int_\R V\right) \int_0^1 \frac{1}{\x'(s)}\,ds + \int_0^1 \x(s)\,ds = E^{(3)}(\xi).
\]
This proves~\eqref{ineq:th:k=3:limsup}.
\end{proof}

\begin{thm}[Case 4, second critical regime: $\beta_n\sim 1$]\label{SCR}
Let $\beta_n>0$ be a sequence such that $\beta_n \to c>0$ as $n\to \infty$. Then, as $n\to \infty$, the functionals $E_n^{(4)}$ defined in \eqref{rewrite-E24-integral} $\Gamma$-converge to the functional $E^{(4)}$ defined in terms of measures $\mu\in \mathcal{M}(\dom)$ by
\begin{equation}
\label{def:E4mu}
E^{(4)}(\mu) := \begin{cases}
\displaystyle c\int_\dom V_{\mathrm{eff}}\Bigl(\frac c{\rho(x)}\Bigr)\rho(x)\, dx
+ \int_\dom x\rho(x)\, dx & \text{if } \mu = \rho dx,\\
+\infty &\text{otherwise,}
\end{cases}
\end{equation}
or in terms of increasing functions $\x\in BV_{loc}(0,1)$ as
\[
E^{(4)}(\x)=c\int_{0}^1V_{\mathrm{eff}}(c\x'(s))\,ds + \int_0^1\x(s)\,ds,
\]
where $V_{\mathrm{eff}}(t):= \sum_{k=1}^{\infty}V(kt)$ for every $t\in \mathbb{R}$,
and $\mu$ and $\xi$ are linked by~\eqref{link:mu-xi}.
In addition, if $E_n^{(4)}(\mu_n)$ is bounded, then $\mu_n$ is weakly compact.
\end{thm}

\begin{proof}
Again the compactness follows from Theorem~\ref{compactness}. We first prove the theorem under the assumption that $c=1$, and comment on the general case at the end.

\textbf{Liminf inequality.}
We will show that for every sequence $(x^n_1,\dots,x^n_n)_n$ of $n$-tuples, converging to $\xi$ in $BV_{\mathrm{loc}}$ in the sense of Theorem~\ref{compactness},
\begin{equation}
\label{ineq:th:k=4:liminf}
\liminf_{n\to\infty} E^{(4)}_n(x^n_1,\dots,x^n_n) \geq  E^{(4)}(\xi).
\end{equation}
Take such a sequence $(x^n_1,\dots,x^n_n)_n$.
We first rewrite the functional $E_n^{(4)}$ in a more convenient way.
For every $k\in\mathbb{N}$ we define the function $V^k(t):=V(kt)$. Hence from \eqref{rewrite-E24-integral} we have
\begin{equation}\label{newQn}
E_n^{(4)}(x^n_1,\dots,x^n_n)= \frac{1}{n}\sum_{k=1}^{n}\sum_{j=0}^{n-k}
V^k\left(\frac{x^n_{i+k}-x^n_i}{\frac{k}{n}}\right) + \frac{1}{n}\sum_{i=0}^n x^n_i.
\end{equation}
The expression in the argument of $V^k$ resembles a {gradient}, and we make the gradient term appear explicitly using the affine interpolant $\x^n$ of the $x_i^n$ (see~\eqref{pa}). For fixed $k\in \{0,\dots,n-1\}$ and $\ell\in \{0,\dots,n-1\}$ and for $i\leq\ell\leq k+i-1$, we have
\begin{align}
\frac{x^n_{i+k}-x^n_i}{\frac{k}{n}} =\frac1k \sum_{m=i-\ell}^{i-\ell+k-1}\frac{x^n_{\ell+m+1}-x^n_{\ell+m}}{\frac{1}{n}} = \frac1k \sum_{m=i-\ell}^{i-\ell+k-1}(\x^n)'\left(s+\frac mn\right)
\end{align}
for every $s\in \left(\frac\ell n, \frac{\ell+1}n\right)$. Then
\begin{align*}
\frac1n V^k\left(\frac{x^n_{i+k}-x^n_i}{\frac{k}{n}}\right) &= \frac1k \sum_{\ell=i}^{i+k-1} \int_{\frac\ell n}^{\frac{\ell+1}n}V^k\left(\frac1k \sum_{m=i-\ell}^{i-\ell+k-1}(\x^n)'\left(s+\frac mn\right)\right)ds\\
&\stackrel{(j=\ell-i)}{=} \frac1k \sum_{j=0}^{k-1} \int_{\frac{j+i} n}^{\frac{j+i+1}n}V^k\left(\frac1k \sum_{m=-j}^{k-1-j}(\x^n)'\left(s+\frac mn\right)\right)ds
\end{align*}
Therefore, we can rewrite the first term in \eqref{newQn} in terms of the function $\tilde{x}$, as
\begin{align*}
\sum_{i=0}^{n-k}\frac1n V^k\left(\frac{x^n_{i+k}-x^n_i}{\frac{k}{n}}\right)=& \frac1k \sum_{i=0}^{n-k} \sum_{j=0}^{k-1} \int_{\frac{j+i} n}^{\frac{j+i+1}n}V^k\left(\frac1k \sum_{m=-j}^{k-1-j}(\x^n)'\left(s+\frac mn\right)\right)ds\\
=& \frac1k\sum_{j=0}^{k-1} \int_{\frac{j} n}^{\frac{n-k+j+1}n}V^k\left(\frac1k \sum_{m=-j}^{k-1-j}(\x^n)'\left(s+\frac mn\right)\right)ds\\
=& \frac1k\sum_{j=0}^{k-1} \int_{\frac{j} n}^{1-\frac{k-j-1}n}V^k\left(\frac1k \sum_{m=-j}^{k-1-j}(\x^n)'\left(s+\frac mn\right)\right)ds,
\end{align*}
and the first term of the functional $E_n^{(4)}$ becomes
\begin{equation}\label{Q_ncont}
\frac{1}{n}\sum_{k=1}^{n}\sum_{i=0}^{n-k}
V^k\left(\frac{x^n_{i+k}-x^n_i}{\frac{k}{n}}\right) = \sum_{k=1}^{n}\frac1k\sum_{j=0}^{k-1} \int_{\frac{j} n}^{1-\frac{k-j-1}n}V^k\left(\frac1k \sum_{m=-j}^{k-1-j}(\x^n)'\left(s+\frac mn\right)\right)ds.
\end{equation}

Now fix $\delta>0$ and note that by Theorem~\ref{compactness}, $\x^n$ converges to $\x$ weakly in $BV(0,1-\delta)$. We claim that for every fixed integer $N$ the following lower bound is satisfied:
\begin{equation}\label{bfa}
\liminf_{n\to+\infty} E_n^{(4)}(x^n_1,\dots,x^n_n) \geq \int_0^{1-\delta} V^N_{\textrm{eff}}(\x'(s))\,ds + \int_0^{1-\delta} \x(s)\,ds,
\end{equation}
where the energy density $V^N_{\textrm{eff}}$ is defined as $V^N_{\textrm{eff}}(t):= \sum_{k=1}^{N} V(kt)$ for every $t\in \mathbb{R}$.
This claim implies the lower bound~\eqref{ineq:th:k=4:liminf} by the arbitrariness of $N$ and of $\delta$.

As in the proofs of earlier theorems we focus on the first term of the discrete energy $E_n^{(4)}$; indeed, since $\x^n \to \x$ in $L^1(0,1-\delta)$ (by the $BV$-convergence), the bound on the second term of the energy in terms of the integral of $\x$ on $(0,1-\delta)$ follows.

Let $N$ be fixed (independent of $n$). Then, for $n\geq N$,
\[
\frac{1}{n}\sum_{k=1}^{n}\sum_{i=0}^{n-k}
V^k\left(\frac{x^n_{i+k}-x^n_i}{\frac{k}{n}}\right)
\geq \sum_{k=1}^{N}\frac1k\sum_{j=0}^{k-1} \int_{\frac{j} n}^{1-\frac{k-j-1}n}V^k\left(\frac1k \sum_{m=-j}^{k-1-j}(\x^n)'\left(s+\frac mn\right)\right)ds.
\]
We note that, since $j$ and $k$ run through a finite set independent of $n$, for every $\eta>0$ 
there exists an integer $\nu(\eta)$ such that, if $n\geq \nu(\eta)$, then
\[
\int_{\frac{j} n}^{1-\frac{k-j-1}n}V^k\left(\frac1k \sum_{m=-j}^{k-j-1}(\x^n)'\left(s+\frac mn\right)\right)ds \geq
\int_{\eta}^{1-\delta -\eta}V^k\left(\frac1k \sum_{m=-j}^{k-j-1}(\x^n)'\left(s+\frac mn\right)\right)ds.
\]
for every $j$ and for every $k$. Moreover, for every $k$ and $j$, also the convex combination
\[
\x^n_{k,j}(s):=\frac1k \sum_{m=-j}^{k-j-1}\x^n\left(s+\frac mn\right)
\]
converges to $\x$ weakly in $BV(\eta,1-\delta - \eta)$, since it is bounded in $W^{1,1}(\eta,1-\delta - \eta)$ and has the same $L^1$-limit as the sequence $\x^n$. Therefore, for every $k$ and $j$
\begin{equation}\label{lscBV}
\liminf_{n\to\infty}\int_{\eta}^{1-\delta -\eta}V^k((\x^n_{k,j})'(s))ds \geq \int_{\eta}^{1-\delta -\eta}V^k(\x'(s))ds,
\end{equation}
since the integral functional in \eqref{lscBV} is lower semicontinuous with respect to the weak convergence in $BV$, by e.g.~\cite[Proposition 5.1--Theorem 5.2]{AmbrosioFuscoPallara00}. In conclusion,
\begin{align*}
\liminf_{n\to\infty}E_n^{(4)}(x^n_1,\dots,x^n_n) &\geq  \sum_{k=1}^{N}\frac1k\sum_{j=0}^{k-1}\int_{\eta}^{1- \delta-\eta}V^k(\x'(s))\,ds + \int_0^{1-\delta} \x(s)\,ds\\
& = \sum_{k=1}^{N}\int_{\eta}^{1-\delta-\eta}V^k(\x'(s))\,ds + \int_0^{1-\delta} \x(s)\,ds\\
&= \int_\eta^{1-\delta-\eta}V^N_{\mathrm{eff}}(\x'(s))\,ds + \int_0^{1-\delta} \x(s)\,ds,
\end{align*}
and the claim \eqref{bfa} follows by the arbitrariness of $\eta$.

\medskip

\textbf{Limsup inequality}. By Theorem \ref{relax} we can reduce to proving the existence of a recovery sequence for a function $\x\in W^{1,1}(0,1)$.

\medskip

Therefore, let $\x \in W^{1,1}(0,1)$ be an increasing function such that $E^{(4)}(\x)<\infty$.
For every $n\in \mathbb{N}$ we define the piecewise affine function $\x^n$ and the points $x^n_i$ by $\x^n\left(\frac in\right):= x^n_i := \x\left(\frac in\right)$.
The sequence $\x^n$ converges to $\x$ strongly in $W^{1,1}$, and therefore also in $BV_{\mathrm{loc}}$.

As in~\eqref{newQn} we  write
\begin{align*}
E^{(4)}_n(x^n_1,\dots,x^n_n)= \frac{1}{n}\sum_{k=1}^{n}\sum_{j=1}^{n-k}
V^k\left(\frac{x^n_{j+k}-x^n_j}{\frac{k}{n}}\right) + \frac{1}{n}\sum_{j=0}^n x^n_j.
\end{align*}
Since the second term of the functional converges to the integral of $\x$ in $(0,1)$, we focus on the first term. As in the proof of the previous theorem the convexity of the function $V^k$ implies, by Jensen's inequality,
\begin{equation}
\frac{1}{n}\sum_{k=1}^{n}\sum_{j=1}^{n-k}
V^k\left(\frac{x^n_{j+k}-x^n_j}{\frac{k}{n}}\right)
=
\frac{1}{n}\sum_{k=1}^{n}\sum_{j=0}^{n-k}
V^k\left( \frac{n}{k}\int_{\frac jn}^{\frac{j+k}{n}}\x'(s)\,ds\right) \leq
\sum_{k=1}^{n}\sum_{j=0}^{n-k} \frac{1}{k}\int_{\frac jn}^{\frac{j+k}{n}}
V^k(\x'(s))\,ds.
\end{equation}
Since
\[
\sum_{j=0}^{n-k}\frac{1}{k}\int_{\frac jn}^{\frac{j+k}{n}} V^k(\x'(s))\,ds \leq \int_0^1V^k(\x'(s))\,ds\quad \mbox{for every }k=1,\dots,n,
\]
we have the following estimate for the energy:
\[
 \frac{1}{n}\sum_{k=1}^{n}\sum_{j=1}^{n-k}
V^k\left(\frac{x^n_{j+k}-x^n_j}{\frac{k}{n}}\right) \leq \sum_{k=1}^{n}\int_0^1V^k(\x'(s))\,ds
\leq \int_0^1 V_{\mathrm{eff}}(\x'(s))\,ds,
\]
which proves the desired inequality,
\[
\limsup_{n\to\infty} E^{(4)}_n(x^n_1,\dots,x^n_n)\leq E^{(4)}(\x).
\]

\medskip
\emph{General $c$}. The case of general $c = \lim_{n\to\infty} \beta_n$ follows by rescaling, as in the proof of Theorem~\ref{th:first-critical}. In terms of the scaled measure
\[
\widetilde\mu_n := \frac1n \sum_{i=1}^n \delta_{\beta_nx_i},
\]
the functional $E_n^{(4)}$ reads
\[
E_n^{(4)}(\mu_n) = \frac {n\beta_n}2\iint_\doms V(x-y)\, \widetilde\mu_n\boxtimes\widetilde\mu_n(dxdy) + \frac1{\beta_n}\int_\dom x\, \widetilde\mu_n(dx).
\]
Since $\beta_n$ is bounded away from zero and infinity, the same arguments apply, and we find that the right-hand side $\Gamma$-converges to
\[
\widetilde E^{(4)}(\widetilde\mu) := \begin{cases}
\displaystyle c\int_\dom V_{\mathrm{eff}}\Bigl(\frac 1{\widetilde\rho(x)}\Bigr)\widetilde\rho(x)\, dx
+ \frac1c\int_\dom x\widetilde\rho(x)\, dx & \text{if } \widetilde\mu = \widetilde\rho dx,\\
+\infty &\text{otherwise,}
\end{cases}
\]
where $\mu$ and $\widetilde\mu$ are linked by
\[
\int_\dom \varphi(x)\, \widetilde\mu(dx) = \int_\dom \varphi(cx)\, \mu(dx)\qquad
\text{for all }\varphi\in C_b(\dom).
\]
Back-transformation gives the $\Gamma$-convergence of the unscaled $E^{(4)}_n$ to the $E^{(4)}$ defined in~\eqref{def:E4mu}.
\end{proof}

\begin{thm}[Case 5, supercritical regime: $\beta_n\to\infty$]
\label{th:supercritical}
Let $\beta_n>0$ be a sequence such that $\beta_n\to \infty$ as $n\to \infty$. Then the functionals $E_n^{(5)}$ defined in \eqref{rewrite-E5} $\Gamma$-converge with respect to the strong convergence in $L^1_{loc}$ to the functional $E^{(5)}$ defined for $\x\in BV_{loc}(0,1)$, $\x$ increasing, as:
\begin{equation}\label{treC}
E^{(5)}(\x) = \begin{cases}
\displaystyle\int_0^1 \xi(s)\,ds &\text{if } \x'(s) \geq 1 \text{ for a.e. }0<s<1,\\
+\infty &\text{otherwise},
\end{cases}
\end{equation}
or equivalently, in terms of measures $\mu$ linked to $\x$ by~\eqref{link:mu-xi},
\[
E^{(5)}(\mu) = \begin{cases}
\displaystyle
\int_\dom x\, \mu(dx) & \text{if } \mu \leq \Lebesgue,\\
+\infty &\text{otherwise}.
\end{cases}
\]
\end{thm}

Note that the inequality $\mu\leq\Lebesgue $ is intended in the sense of measures, i.e. $\mu(A)\leq \Lebesgue(A)$ for all $A\subset\dom$ measurable. Equivalently, one can require that $\mu\ll\Lebesgue $ and $\d\mu/\d\Lebesgue \leq 1$.

\begin{proof}
First of all, we define the sequence $\alpha_n:=\frac{1}{2\pi}\log\left(\frac{2}{\pi}\beta_n^2\right)$ and rewrite the energy \eqref{rewrite-E5} in terms of the new sequence, as
\begin{equation}\label{rewrite-E5_2}
E_n^{(5)}(x_1,\dots, x_n) = \frac{\pi}{2}\,\frac{e^{2\pi\alpha_n}}{n\alpha_n} \sum_{k=1}^n\sum_{j=1}^{n-k} V\left(n\alpha_n(x_{j+k}-x_j)\right) + \frac1n \sum_{j=1}^n x_j.
\end{equation}
Notice that $\beta_n^2 = \frac{\pi}{2}e^{2\pi\alpha_n}$ and that $\alpha_n\gg1$, since $\beta_n\gg1$.

\medskip

\textbf{Liminf inequality.}
Let $(x_1^n,\dots,x_n^n)$ be a sequence of $n$-vectors such that the piecewise affine interpolation $\x^n$, as defined in Theorem~\ref{compactness} converges in $BV_{\mathrm{loc}}(0,1)$ to some $\xi$.
As for the other cases, the second term in the discrete energy $E^{(5)}_n$ is lower semi-continuous with respect to this convergence, and therefore we focus on the first term.

The energy satisfies the trivial estimate
$$
\sum_{k=1}^n\sum_{i=0}^{n-k}V\left(n\alpha_n(x^n_{i+k}-x^n_i)\right) \geq
\sum_{i=0}^{n-1}V\left(n\alpha_n(x^n_{i+1}-x^n_i)\right),
$$
since every term in the sum is nonnegative. Now, following the proof of the liminf inequality in Theorem \ref{SCR} we can write
$$
\frac1n V\left(\alpha_n\,\frac{x^n_{i+1} - x^n_i}{\frac1n} \right) = \int_{\frac in}^{\frac{i+1}{n}} V(\alpha_n(\x^n)'(s)) ds,
$$
and therefore
\[
\frac1n\sum_{i=0}^{n-1}V\left(n\alpha_n(x^n_{i+1}-x^n_i)\right) = \int_0^1 V(\alpha_n(\x^n)'(s)) ds.
\]

Hence, we have the following bound for the first term of the energy:
\begin{align*}
\frac{\pi}{2}\frac{e^{2\pi \alpha_n}}{n\alpha_n}\sum_{k=1}^n\sum_{i=0}^{n-k}V\left(n\alpha_n(x^n_{i+k}-x^n_i)\right) \geq \int_0^1 V_n((\x^n)'(s))ds,
\end{align*}
where $V_n(t):= \frac{\pi}{2}\frac{e^{2\pi \alpha_n}}{\alpha_n}V(\alpha_n t)$. We claim that for any $t_n\to t$,
\begin{equation}
\label{ineq:Vn}
\liminf_{n\to\infty}V_n(t_n) \geq V_{\infty}(t):=
\begin{cases}
+\infty & \mbox{if }\, 0\leq t<1,\\
0 & \mbox{if }\, t \geq 1.
\end{cases}
\end{equation}
By e.g.~\cite[Prop.~2.2]{BraidesGelli02a}, this inequality implies that
\[
\liminf_{n\to\infty} E_n^{(5)}(x_1^n,\dots,x_n^n) \geq E^{(5)}(\x).
\]

To prove~\eqref{ineq:Vn}, we only need to show that if $0\leq t<1$, then $V_n(t_n)\to\infty$.
This can be easily proved in the following way. Since $0\leq t<1$ can be rewritten as $t:=1-\eta$ (for $0<\eta\leq 1$), then for a sequence $t_n$ converging to $t$, we can
 assume that $|t_n-t|\leq \frac{\eta}{2}$, and in particular $t_n\leq 1-\frac{\eta}2$.
Since  $V_n$ is decreasing for every $n$, we have the bound $V_n(t_n)\geq V_n\left(1-\frac{\eta}{2}\right)$ for $n$ large enough.
Therefore
$$
\liminf_{n\to\infty}V_n(t_n) \geq \lim_{n\to\infty} V_n\left(1-\frac{\eta}{2}\right)=\infty,
$$
which concludes the proof.

\medskip

\textbf{Limsup inequality.} We can once more invoke Theorem \ref{relax}, since $f(t)=(\mbox{sc}^- V_{\infty})(t)$ satisfies the assumptions of the theorem, and
construct a recovery sequence only for increasing functions $\x \in W^{1,1}(0,1)$ such that $E^{(5)}(\x)<\infty$.

By density we can further reduce to $\x$ piecewise affine.

\medskip

Let us first assume that $\x$ is linear, i.e., $\x(s)=(1+\ell)s$, for $\ell\geq0$ (since $E^{(5)}(\x)<\infty$).

The recovery sequence $x_i^n$ can be constructed in the following way.
Let $\delta_n>0$ be a sequence such that $\delta_n\to \ell$; then we set $x_i^n:= (1+\delta_n)\frac in$ for $i=0,\dots,n$. The sequence $(x_i^n)$ is increasing, $x_0^n=0$ for every $n$, and
\begin{align}\label{E5sup}
E^{(5)}_n(x^n_1,\dots,x^n_n) =& \frac{\pi}{2}\frac{e^{2\pi\alpha_n}}{n\alpha_n}\sum_{k=1}^n\sum_{i=0}^{n-k}V\left(n\alpha_n(x^n_{i+k}-x^n_i)\right) + \frac1n\sum_{i=0}^n x^n_i \nonumber\\
=& \frac{\pi}{2}\frac{e^{2\pi\alpha_n}}{n\alpha_n}\sum_{k=1}^n\sum_{i=0}^{n-k}V\left(\alpha_nk(1+\delta_n)\right) + \frac{(1+\delta_n)}{n^2}\sum_{i=0}^n i.
\end{align}
We claim that the sequence $\delta_n$ can be chosen so that
\begin{equation}\label{claim5}
\lim_{n\to\infty} \frac{e^{2\pi\alpha_n}}{n\alpha_n}\sum_{k=1}^n\sum_{i=0}^{n-k}V\left(\alpha_nk(1+\delta_n)\right) = 0.
\end{equation}
We focus on the term $k=1$ in the sum in \eqref{claim5}; we note that it is an upper bound for the other terms in the sum, corresponding to $k\geq 2$.
Using the form of $V$ for large values of $t$ recalled above, we can rewrite the term $k=1$ as

\begin{align*}
\frac{e^{2\pi\alpha_n}}{n\alpha_n}\sum_{i=0}^{n-1}V\left(\alpha_n(1+\delta_n)\right) &\sim \frac{e^{2\pi\alpha_n}}{n\alpha_n}n (\alpha_n(1+\delta_n))e^{-2\pi(\alpha_n(1+\delta_n))}\\
&= (1+\delta_n)e^{-2\pi\alpha_n\delta_n},
\end{align*}
up to a remainder going to zero exponentially fast as $n\to \infty$.
Clearly, the sequence $\delta_n$ can always be chosen so that the last expression converges to zero (indeed, even if $\ell=0$ we can choose $\delta_n \to 0$ such that $\alpha_n\delta_n\to \infty$).

Therefore, the claim \eqref{claim5} follows directly, since every other term in the sum is estimated by the term $k=1$; in conclusion,
$$
\limsup_{n\to\infty}E_n^{(5)}(x^n) \leq \frac12(1+\ell) = \int_0^1x(s)ds = E^{(5)}(x).
$$
To illustrate the general case of $x$ piecewise affine we can reduce to $x$ of the form
\begin{equation}
x(t)=\begin{cases}
\medskip
(1+\ell_1)t \quad &\mbox{if }\,t\leq t^* \\
(1+\ell_2)\left(t-t^*\right) + (1+\ell_1)t^* \quad &\mbox{if }\,t\geq t^*,
\end{cases}
\end{equation}
where, for example $0<\ell_1<\ell_2$, and $t^*\in (0,1)$. Then assuming that $nt^*\in \mathbb{N}$, the approximating sequence is defined as
\begin{equation}
x^n_i=\begin{cases}
\medskip
(1+\ell_1)\frac{i}{n} \quad &\mbox{if }\,0\leq i\leq nt^* \\
(1+\ell_2)\left(\frac{i}{n}- t^*\right) + (1+\ell_1)t^* \quad &\mbox{if }\,nt^*\leq i\leq n.
\end{cases}
\end{equation}
Clearly $x_i^n$ converges to $x$ in $L^1$.

We claim that for every $i$ and $k$
\begin{equation}\label{Est:slope}
x_{i+k}^n - x_i^n \geq (1+\ell_1)\frac{k}{n}.
\end{equation}
Before proving \eqref{Est:slope} we notice that it implies the required bound for the energy. Indeed, it allows us to reduce to the case of one
single slope, which was already treated in the first part of the proof.

It remains to prove \eqref{Est:slope}. We first notice that if $i, i+k \leq nt^*$ or $i, i+k \geq nt^*$, then \eqref{Est:slope} follows immediately, since $\ell_1<\ell_2$. So we assume that $i< nt^* \leq i+k$. Then
\begin{align*}
x_{i+k}^n - x_i^n &= (1+\ell_2)\left(\frac{i}{n}- t^*\right) + (1+\ell_1)t^* - (1+\ell_1)\frac{i}{n}\\
&= (\ell_2-\ell_1)\frac in + (1+\ell_2)\frac{k}n - (\ell_2 - \ell_1)t^*\\
&\geq (\ell_2-\ell_1)\frac in + (1+\ell_2)\frac{k}n - (\ell_2 - \ell_1)\left(\frac{i+k}{n}\right),
\end{align*}
where the last inequality follows from the assumption $i< nt^* \leq i+k$. The last term is exactly $(1+\ell_1)\frac{k}{n}$, and therefore the claim is proved.
\end{proof}

\section{Uniqueness and convergence of minimizers}\label{sec:uniqueness}

We now prove Theorem~\ref{lem:uniqueness} and Corollary~\ref{cor:minimizers-converge}.

\begin{proof}[Proof of Theorem~\ref{lem:uniqueness}]
To show existence, we take each of the cases $k=\{1,2,3,4,5\}$ in turn and consider a minimizing sequence $\mu_m$ for each of them. Note that we only need to prove compactness, since each of the $\Gamma$-limits is automatically lower semicontinuous. For $E^{(2)}$, $E^{(3)}$, $E^{(4)}$, and $E^{(5)}$, the compactness is immediate, since boundedness of $E^{(k)}(\mu_m)$ implies boundedness of $\int_\dom x\, d\mu_m(x)$, and therefore tightness. For $E^{(1)}$ we use the same argument as in the proof of Theorem~\ref{th:subcritical} to show that $\int_\dom x\,\mu_m(dx)$ is bounded

We prove uniqueness by proving strict convexity, either in $\mu$ or in $\x$.
Note that convexity of $\x\mapsto E^{(k)}(\x)$ corresponds to displacement convexity. Also note that the term $\int_\dom x\mu = \int_0^1 \x(s)\, ds$ is convex in both senses, and therefore we only need to prove strict convexity of the interaction terms in each of the $E^{(k)}$.

We treat the cases separately. The functional $\mu\mapsto E^{(3)}(\mu)$ is strictly convex because the function $\rho\mapsto\rho^2$ is strictly convex. Similarly, since $s\mapsto V_{\mathrm{eff}}(s)$ is strictly convex for $s>0$, the function $\x\mapsto E^{(4)}(\x)$ is strictly convex. Writing $E^{(2)}$ as
\[
E^{(2)}(\x) = \frac c2 \int_0^1\int_0^1 V\bigl(c(\x(t)-\x(s))\bigr)\, dsdt + \int_0^1 \x(s)\, ds
=  c\int_0^1\int_0^t V\bigl(c(\x(t)-\x(s))\bigr)\, dsdt + \int_0^1 \x(s)\, ds.
\]
Since $s\mapsto V(s)$ is strictly convex for $s>0$, it follows that $E^{(2)}$ is strictly convex. For $E^{(1)}$ a similar argument applies.

The functional $E^{(5)}$ is non-strictly convex; but one finds in a straightforward manner that $\mu = \Lebesgue\bigr|_{[0,1]}$ is the unique minimizing measure.
\end{proof}

\begin{proof}[Proof of Corollary~\ref{cor:minimizers-converge}]
Fix the case $k\in\{1,2,3,4,5\}$.
Given a sequence of minimizers $(x_1^n,\dots,x_n^n)$ of $E_n^{(k)}$, we again set $\mu_n = \frac1n \sum_{i=1}^n \delta_{x_i^n}$. Taking $\mu$ to be the unique minimizer of $E^{(k)}$ given by Theorem~\ref{lem:uniqueness}, by Theorem~\ref{th:main} there exists a recovery sequence $\widehat \mu_n\weakto \mu$ along which $E_n^{(k)}(\widehat\mu_n)$ converges and therefore remains bounded. Since $\mu_n$ are minimizers, $E_n^{(k)}(\mu_n)\leq E_n^{(k)}(\widehat \mu_n)$ also remains bounded. By the compactness statement of Theorem~\ref{th:main}, this imples that $\mu_n$ is compact.

Therefore $\mu_n$ converges along a subsequence $n_k$ to a limit $\widetilde \mu$;  the minimality of $\mu_{n_k}$ transfers to $\widetilde \mu$, so that $\widetilde \mu$ is also a minimizer of the limit $E^{(k)}$ (see e.g.~\cite[Corollary~7.20]{DalMaso93}). By Theorem~\ref{lem:uniqueness} this minimizer is unique, $\widetilde \mu = \mu$,  and the whole sequence $\mu_n$ converges.
\end{proof}



\appendix
\section{}

\subsection{$V$ as a convolution}\label{convolution}
In this section we prove that V can be written as a convolution. Our definition of the Fourier transform and its inverse is
\begin{align*}
\widehat{f}(\xi):= \int_{\mathbb{R}}e^{-2\pi i \xi x}f(x)\,dx,\qquad
\check{f}(x):= \int_{\mathbb{R}}e^{2\pi i x \xi}f(\xi)\,d\xi.
\end{align*}
We claim that $V = W\ast W$, where $W:= \check{U}$,  $U:=\sqrt{\widehat{V}}$, and $\widehat V\geq0$.

\medskip

First of all we notice that, if the function $W$ is well-defined, then $V=W\ast W$. Indeed, since
\begin{align*}
V = W\ast W \Leftrightarrow \hat{V} = \widehat{W\ast W},
\end{align*}
by elementary properties of the Fourier transform and by the definition of $W$ and $U$ we have
\begin{equation*}
\hat{V} = \widehat{W\ast W} = \hat{W}^2 = U^2,
\end{equation*}
which proves the claim.

\bigskip

Now we prove that the function $W$ is well-defined.

\medskip

We first observe that since $V$ is even, $\hat{V}$ is real-valued. Moreover, the Fourier transform of $V$ can be computed explicitly, as we now show.

\subsubsection{Fourier transform of $V$}

We start computing the Fourier transform of the function $\varphi(x)=\frac{x}{\sinh^2(\pi x)}=  -V'(x)$.
We can rewrite the function $\varphi$ as
\begin{equation}
\varphi(x)=-\frac{x}{\pi}\,\frac{d}{dx}\coth(\pi x).
\end{equation}
In order to compute the Fourier transform of $\varphi$, we first note that $\varphi$ has a non-integrable singularity at the origin; therefore $\varphi$ should be considered a tempered distribution, whose effect on a Schwartz function $\psi$ is defined as the principle-value integral
\[
\langle \varphi,\psi\rangle := - PV \int_\R \frac x\pi \Bigl(\frac d{dx} \coth(\pi x)\Bigr)\psi(x)\, dx = \frac1\pi PV \int_\R \coth(\pi x) \frac d{dx} (x\psi(x))\, dx.
\]
Below we show that
\begin{equation}\label{FT:coth}
\widehat{\coth(\pi\cdot)}(\xi) = -i \coth(\pi \xi),
\end{equation}
from which it follows by the properties of Fourier transforms that
\[
\widehat \varphi(\xi) = -\frac i\pi \frac d{d\xi}\bigr( \xi \coth(\pi\xi)\bigr)
\qquad\text{and}\qquad \widehat V(\xi) =  \frac{1}{2\pi^2\xi}\frac d{d\xi}\bigr( \xi \coth(\pi\xi)\bigr).
\]
From the explicit expression of $\hat{V}$ we can see that it is nonnegative. Since $\hat{V}$ is an even function, it is sufficient to
check that it is positive for $\xi>0$. Writing more explicit the last term in \eqref{FT:coth} we have
\begin{align*}
\partial_{\xi} \Big(\xi \coth(\pi \xi) \Big) = \coth(\pi\xi) - \frac{\pi\xi}{\sinh^2(\pi \xi)} =
\frac{1}{\sinh(\pi\xi)}\left(\cosh(\pi\xi) - \frac{\pi\xi}{\sinh(\pi\xi)} \right).
\end{align*}
And, since $\sinh t \geq t$ for $t>0$, from the previous expression we get
\begin{align}
\partial_{\xi} \Big(\xi \coth(\pi \xi) \Big) \geq \frac{1}{\sinh(\pi\xi)}\left(\cosh(\pi\xi) - 1 \right),
\end{align}
which is positive for $\xi >0$ since $\cosh t \geq 1$ for every $t$. Then this proves that $\hat{V}(\xi)>0$ for $\xi>0$.

\medskip

This proves in turn that the function $U$ is well defined, and it is even, since $\hat{V}$ is even.
Hence the function $W$ is real (once again, since $W$ is defined as the inverse Fourier transform of $\check{U}$ and $\check{U}$ has a non-integrable singularity at the origin, $\check{U}$ should be considered a tempered distribution).

\bigskip
To prove~\eqref{FT:coth} it is convenient to use the representation of the hyperbolic cotangent in terms of a series, i.e.,
$$
\coth x = x \sum_{k=-\infty}^{\infty}\frac{1}{k^2\pi^2 + x^2} = \frac{1}{x} + 2\sum_{k= 1}^{\infty}\frac{x}{k^2\pi^2 + x^2};
$$
then we have
$$
\coth(\pi x) = \frac{1}{\pi x} + 2\pi\sum_{k= 1}^{\infty}\frac{x}{k^2\pi^2 + \pi^2 x^2}.
$$
Therefore, using the previous expression, we have
\begin{align}\label{FT:2-3}
\widehat{\coth(\pi\cdot)}(\xi) = \int_{\mathbb{R}}\frac{e^{-2\pi i \xi x}}{\pi x}\,dx +
2\pi\sum_{k=1}^{\infty}\int_{\mathbb{R}}e^{-2\pi i \xi x}\,\frac{x}{k^2\pi^2+\pi^2x^2}\,dx.
\end{align}
We compute the two integrals in the right-hand side of the previous expression separately. For the first term we have
\begin{equation}\label{FT:2}
\int_{\mathbb{R}}\frac{e^{-2\pi i \xi x}}{\pi x}\,dx = -i \,\textrm{sgn}(\xi).
\end{equation}
For the second term we claim that
\begin{equation}\label{FT:3}
\int_{\mathbb{R}}e^{-2\pi i \xi x}\,\frac{x}{k^2\pi^2+\pi^2x^2}\,dx = \frac{i}{\pi}\, e^{-2k\pi |\xi|} \,\textrm{sgn}(\xi).
\end{equation}
To prove this we need some preliminary steps. First of all, elementary calculations show that
\begin{equation}\label{FT:3a}
\int_{\mathbb{R}}e^{-2\pi i \xi x}\,\frac{x}{k^2\pi^2+\pi^2x^2}\,dx = \frac{i}{2\pi} \partial_\xi \int_{\mathbb{R}}e^{-2\pi i \xi x}\,\frac{1}{k^2\pi^2+\pi^2x^2}\,dx.
\end{equation}
Moreover, the Fourier transform of the function $f$ defined as
$$
f(x):= \frac{2a}{a^2+(2\pi x)^2}
$$
is $\hat{f}(\xi) = e^{-a|\xi|}$; using this formula in our case leads to
\begin{equation}\label{FT:3b}
\int_{\mathbb{R}}e^{-2\pi i \xi x}\,\frac{1}{k^2\pi^2+\pi^2x^2}\,dx = \frac{1}{\pi k}\, e^{-2\pi k|\xi|}.
\end{equation}
Therefore, combining \eqref{FT:3a} and \eqref{FT:3b} we have
$$
\int_{\mathbb{R}}e^{-2\pi i \xi x}\,\frac{x}{k^2\pi^2+\pi^2x^2}\,dx = \frac{i}{2\pi} \partial_\xi \frac{1}{\pi k}\, e^{-2\pi k|\xi|} = - \frac{i}{\pi}\,\textrm{sgn}(\xi)\,e^{-2\pi k|\xi|},
$$
which proves the claim \eqref{FT:3}. Finally, from relations \eqref{FT:2-3}, \eqref{FT:2} and \eqref{FT:3} we have
\begin{align}\label{FT:4}
\widehat{\coth(\pi\cdot)}(\xi) = -i \,\textrm{sgn}(\xi) - 2\, i \,\textrm{sgn}(\xi)\sum_{k=1}^{\infty}e^{-2\pi k|\xi|}.
\end{align}

At this point we make use of the expression of the hyperbolic cotangent in terms of an infinite series, for negative values of its argument, i.e.,
$$
\coth x = -1 -2\sum_{k=1}^\infty e^{2kx}, \quad x<0;
$$
then \eqref{FT:4} reduces simply to
\begin{equation}\label{FT:coth2}
\widehat{\coth(\pi\cdot)}(\xi) = i \,\textrm{sgn}(\xi)\coth(\pi(-|\xi|)) = -i \coth(\pi \xi).
\end{equation}

\medskip

\bigskip
\bigskip
\centerline{\textsc{Acknowledgments}}
\bigskip
\noindent
The authors would like to thank Jan Zeman, Andrea Braides, and Massimiliano Morini for their fruitful comments on the paper. The research of L. Scardia was carried out under the project number M22.2.09342 in the framework of the Research Program of the Materials innovation institute (M2i) (www.m2i.nl). The research of M. A. Peletier has received funding from the ITN ``FIRST'' of the Seventh Framework Programme of the European Community (grant agreement number 238702).


\addcontentsline{toc}{chapter}{References}
\begin{thebibliography}{9}

\bibitem{AliCic} Alicandro R. and Cicalese M.: A General Integral Representation Result for Continuum Limits of Discrete Energies with Superlinear Growth. \textit{SIAM J. Math. Anal.}, \textbf{36}/1 (2004), 1--37.

\bibitem{ACP} Alicandro R., Cicalese M. and Ponsiglione M.: Variational equivalence between Ginzburg-Landau, XY spin systems and
screw dislocation energies. \textit{Indiana Univ. Math. J.} To appear.

\bibitem{AmbrosioFuscoPallara00} Ambrosio L., Fusco N. and Pallara D.: Functions of Bounded Variation and Free Discontinuity Problems. Oxford Mathematical Monographs. Oxford University Press, first edition, 2000.

\bibitem{AmbrosioGigliSavare05} Ambrosio L., Gigli N. and Savar\'e G.: Gradient Flows in Metric Spaces and in the Space of Probability Measures. Lectures in mathematics ETH Z{\"u}rich. Birkh{\"a}user, 2005.

\bibitem{AsaroRice77} Asaro R.J. and Rice J.R.: Strain localization in ductile single crystals. \textit{J. Mech. Phys. Solids}, \textbf{25} (1977), 309-338.

\bibitem{BaskaranMesarovicetc10} Baskaran R., Sreekanth A., Mesarovic S. and Zbib H.M.: Energies and distributions of dislocations in stacked pile-ups.
\textit{Int. J. Solids Struc.}, \textbf{47} (2010), 1144--1153.

\bibitem{Billingsley99} Billingsley P.: Convergence of probability measures. Wiley-Interscience, 1999.

\bibitem{BlakeZisserman} Blake A. and Zisserman A.: Visual Reconstruction. The MIT Press, Cambridge, Massachussets, 1987.

\bibitem{BraidesCicalese07} Braides A. and Cicalese M.: Surface energies in nonconvex discrete systems. \textit{Math. Models Methods Appl. Sci.},
\textbf{17} (2007), 985--1037.

\bibitem{BraidesDalMasoGarroni99} Braides A., Dal Maso G. and Garroni A.: Variational formulation of softening phenomena in fracture mechanics:
the one-dimensional case. \textit{Arch. Ration. Mech. Anal.}, \textbf{146} (1999), 23--58.

\bibitem{BraidesGelli02a} Braides A. and Gelli M.S.: Continuum limits of discrete systems without convexity hypotheses. \textit{Math. Mech. Solids},
\textbf{7}/1 (2002), 41--66.

\bibitem{BraidesMaria06} Braides A. and Gelli M.S.: From discrete systems to continuous variational problems: an introduction.
\textit{Topics on Concentration Phenomena and Problems with Multiple Scales} (2006), 3--77.

\bibitem{BLO} Braides A., Lew A. and Ortiz M.: Effective cohesive behavior of layers of interatomic planes. \textit{Arch. Ration. Mech. Anal.},
\textbf{180} (2006), 151--182.

\bibitem{DalMaso93} Dal Maso G.: An Introduction to $\Gamma$-Convergence. Birkh{\"a}user, 1993.

\bibitem{DeGeusPeerlingsHirschberger11TR} Geus, de T.W.J., Peerlings R.H.J. and Hirschberger C.B.: An analysis of the
pile-up of infinite walls of edge dislocations. Submitted paper.

\bibitem{DengEl-Azab09} Deng J. and El-Azab A.: Mathematical and computational modelling of correlations in dislocation dynamics.
\textit{Model. Simul. Mater. Sci. Eng.}, \textbf{17} (2009).

\bibitem{El-Azab00} El-Azab A.: Statistical mechanics treatment of the evolution of dislocation distributions in single crystals.
\textit{Phys. Rev. B}, \textbf{61}/18 (2000), 11956--11966.

\bibitem{EshelbyFrankNabarro51} Eshelby J.D., Frank F.C. and Nabarro F.R.N.: The equilibrium of linear arrays of dislocations.
\textit{Philos. Mag.}, \textbf{43} (1951), 351--364.

\bibitem{EversBrekelmansGeers04} Evers L.P., Brekelmans W.A.M. and Geers M.G.D.: Scale dependent crystal plasticity framework with dislocation density
and grain boundary effects. \textit{Int. J. Solids Struc.}, \textbf{41} (2004), 5209--5230.

\bibitem{FleckMullerAshbyHutchinson94} Fleck N.A., Muller G.M., Ashby M.F. and Hutchinson J.W.: Strain gradient plasticity: theory and experiment.
\textit{Acta Metall. Mater.}, \textbf{42}/2 (1994), 475--487.

\bibitem{Groma97} Groma I.: Link between the microscopic and mesoscopic length-scale description of the collective behavior of dislocations.
\textit{Phys. Rev. B}, \textbf{56}/10 (1997), 5807--5813.

\bibitem{GromaBalogh99} Groma, I. and Balogh, P.: Investigation of dislocation pattern formation in a two-dimensional self-consistent field approximation.
\textit{Acta Mater.}, \textit{47}/13 (1999), 3647--3654.

\bibitem{GromaCsikorZaiser03} Groma I., Csikor F.F. and Zaiser M.: Spatial correlation and higher-order gradient terms in a continuum description of dislocation
dynamics. \textit{Acta Mater.}, \textbf{51} (20030, 1271--1281.

\bibitem{Hall11} Hall C.L.: Asymptotic analysis of a pile-up of edge dislocation walls. \textit{Mat. Sci. Eng. A - Struct.} (2011), To appear.

\bibitem{Hall51a} Hall E.O.: The deformation and ageing of mild steel: III. Discussion of results.	\textit{Proc. Phys. Soc. Sect. B}, \textbf{64} (1951).

\bibitem{HeadLouat55} Head A.K. and Louat N.: The distribution of dislocations in linear arrays. \textit{Aust. J. Phys.}, \textbf{8}/1 (1955), 1--7.

\bibitem{Hill66} Hill R.: Generalized constitutive relations for incremental deformation of metal crystals by multislip. \textit{J. Mech. Phys. Solids},
\textbf{14} (1966), 95--102.

\bibitem{HillHavner82} Hill R. and Havner K.S.: Perspectives in the mechanics of elastoplastic crystals. \textit{J. Mech. Phys. Solids}, \textbf{30}/5 (1982).

\bibitem{HillRice72} Hill R. and Rice J.R.: Constitutive analysis of elastic-plastic crystals at arbitrary strain. \textit{J. Mech. Phys. Solids},
\textbf{20} (1972), 401--413.

\bibitem{HullBacon01} Hull D. and Bacon D.J.: Introduction to dislocations. Butterworth-Heinemann, 2001.

\bibitem{LimkumnerdVan-der-Giessen08} Limkumnerd S. and Van der Giessen E.: Statistical approach to dislocation dynamics:
From dislocation correlations to a multiple-slip continuum theory of plasticity. \textit{Phys. Rev. B}, \textbf{77}/18 (2008).

\bibitem{McCann97} McCann R.J.: A Convexity Principle for Interacting Gases. \textit{Adv. Math.}, \textbf{128} (1997), 153--179.

\bibitem{MesarovicBaskaranetc10} Mesarovic S., Baskaran R. and Panchenko A.: Thermodynamic coarsening of dislocation mechanics and the
size-dependent continuum crystal plasticity. \textit{J. Mech. Phys. Solids}, \textbf{58}/3 (2010), 311--329.

\bibitem{Petch53} Petch N.J.: The cleavage strength of polycrystals. \textit{J. Iron Steel Institute}, \textbf{174}/1 (1953), 25--28.

\bibitem{RoyAcharya06} Roy A. and Acharya A.: Size effects and idealized dislocation microstructure at small scales: prediction of a phenomenological
model of Mesoscopic Field Dislocation Mechanics. Part II. \textit{J. Mech. Phys. Solids}, \textbf{54}/8 (2006), 1711-1743.

\bibitem{RoyPeerlingsGeersKasyanyuk08} Roy A., Peerlings R.H.J., Geers M.G.D. and Kasyanyuk Y.: Continnum modeling of dislocation interactions: Why discreteness matters? \textit{Mat. Sci. Eng. A - Struct.}, \textbf{486} (2008), 653--661.

\bibitem{GPPSMech} Scardia L., Peerlings R.H.J., Peletier M.A. and Geers M.G.D.: Mechanics of dislocation pile-ups: A unification of scaling regimes.
Submitted paper.

\bibitem{SandierSerfaty04} Sandier E. and Serfaty S.: Gamma-convergence of gradient flows with applications to Ginzburg-Landau. \textit{Comm. Pure Appl. Math.},
\textbf{57}/12 (2004), 1627--1672.

\bibitem{SandierSerfaty10TR} Sandier E. and Serfaty S.: From the Ginzburg-Landau model to vortex lattice problems. Arxiv preprint arXiv:1011.4617 (2010).

\bibitem{ZaiserMiguelGroma01} Zaiser M., Miguel M.C. and Groma I.: Statistical dynamics of dislocation systems: The influence of dislocation-dislocation correlations. \textit{Phys. Rev. B}, \textbf{64}/22 (2001).

\end{thebibliography}
\end{document}